\documentclass[12pt]{amsart}
\usepackage[foot]{amsaddr}

  \usepackage{xhfill} 
  \usepackage{amsthm}
  \usepackage{amssymb}
  \usepackage{amsmath}
  \usepackage{thmtools}
  \usepackage{mathtools}
  \usepackage{wasysym}
  \usepackage[normalem]{ulem}
  \usepackage[main=english,french]{babel}
  \usepackage{csquotes}
  \usepackage{enumitem}
  \usepackage{pdfpages}
  \usepackage{stmaryrd}
  \usepackage{graphicx}
  
  \usepackage{subcaption}
  \usepackage{tikz}
  \usetikzlibrary{shapes,fit,positioning,calc,matrix,arrows,decorations.pathreplacing}
  \tikzset{
    optree/.style={scale=.5,thick,grow'=up,level distance=10mm,inner sep=1pt},
    comp/.style={draw=none,circle,fill,line width=0,inner sep=0pt},
    dot/.style={draw,circle,fill,inner sep=0pt,minimum width=3pt},
    circ/.style={draw,circle,inner sep=1pt,minimum width=4mm},
    emptycirc/.style={draw,circle,inner sep=1pt,minimum width=2mm},  
    root/.style={level distance=10mm,inner sep=1pt},
    leaf/.style={draw=none,circle,fill,line width=0,inner sep=0pt},
    nodot/.style={draw,circle,inner sep=1pt},
  }
  \usepackage{tikz-cd}
  \usepackage[babel=true, kerning=true]{microtype}
  \usepackage[top=0.80in, bottom=1in, left=1in, right=1in]{geometry}
  
  \usepackage{dsfont}
  
  \usepackage[hidelinks]{hyperref}
  \usepackage{cleveref}
  \newtheorem*{theorem*}{Theorem}
  \newtheorem*{cor*}{Corollary}
  \newtheorem*{prop*}{Proposition}
  \newtheorem*{conj*}{Conjecture}

  \newtheorem{theorem}{Theorem}[section]
  \newtheorem{lemma}[theorem]{Lemma}
  \newtheorem{prop}[theorem]{Proposition}
  \newtheorem{cor}[theorem]{Corollary}
  \newtheorem{conj}[theorem]{Conjecture}
  \newtheorem{notation}[theorem]{Notation}

  \theoremstyle{definition}
  \newtheorem*{definition*}{Definition}

  \newtheorem{definition}[theorem]{Definition}

  \newtheorem{remark}[theorem]{Remark}
  
  \newcommand{\ra}{\rightarrow}

  \DeclareMathOperator{\End}{End}           
  \DeclareMathOperator{\Hom}{Hom}           
  \DeclareMathOperator{\Aut}{Aut}           
  \DeclareMathOperator{\id}{id}             
  \DeclareMathOperator{\Spec}{Spec}         

  \newcommand{\Eig}{\mathrm{Eig} \xspace}

  \newcommand{\CH}{\mathrm{CH} \xspace}
  \newcommand{\Rz}{R_Z}
  \newcommand{\Rinf}{R_{\infty}}

  \newcommand{\bC}{\mathbb{C} \xspace}
  \newcommand{\bR}{\mathbb{R} \xspace}
  \newcommand{\bQ}{\mathbb{Q} \xspace}
  \newcommand{\bF}{\mathbb{F} \xspace}
  \newcommand{\Fpbar}{\overline{\mathbb{F}}_{p} \xspace}
  \newcommand{\bZ}{\mathbb{Z} \xspace}
  \newcommand{\bN}{\mathbb{N} \xspace}
  \newcommand{\Qp}{\mathbb{Q}_p \xspace}
  
  \newcommand{\Zp}{\mathbb{Z}_p \xspace}
  \newcommand{\Ql}{\mathbb{Q}_{\ell} \xspace}

\newcommand{\frk}{\mathfrak}
\newcommand{\frkp}{\frk p}

\newcommand{\Gal}{\mathrm{Gal} \xspace}
\newcommand{\Frob}{\mathrm{Frob} \xspace}

\newcommand{\Fix}{\mathrm{Fix} \xspace}
\newcommand{\Stab}{\mathrm{Stab} \xspace}

\newcommand{\bUn}{\mathds 1 \xspace}
\newcommand{\Sym}{\mathrm{Sym} \xspace}
\newcommand{\rk}{\mathrm{rk} \xspace}
\newcommand{\Tr}{\mathrm{Tr} \xspace}

\newcommand{\mcT}{\mathcal{T} \xspace}

\newcommand{\Ker}{\mathrm{Ker} \xspace}

\newcommand{\Mot}{\mathcal{M}\mathrm{ot}}
\newcommand{\Lmot}{\mathcal{LM}\mathrm{ot}}
\newcommand{\h}{\frk h}

\newcommand{\cl}{\mathrm{cl} \xspace}

  \newcommand{\CM}{\mathrm{CM}}

  \newcommand{\mot}{\mathrm{mot}}
  \newcommand{\lef}{\mathrm{Lef}}
  \newcommand{\num}{\mathrm{num}}
  \newcommand{\eqhom}{\mathrm{hom}}
  \newcommand{\ab}{\mathrm{ab}}

  \newcommand{\Vect}{\mathrm{Vect}}

  \newcommand{\Rep}{\mathrm{Rep} \xspace}
  \newcommand{\HS}{\mathrm{HS} \xspace}
  \newcommand{\Isoc}{\mathrm{Isoc} \xspace}
  \newlength{\depthofsumsign}
  \newlength\myccheight

  \tikzset{cross/.style={cross out, draw=black, minimum size=2*(#1-\pgflinewidth), inner sep=0pt, outer sep=0pt},cross/.default={1pt}}

\newcounter{todocount} 

\usepackage[backend=bibtex,url=false,isbn=false]{biblatex}
\begin{document}

\title{Standard conjecture of Hodge type for powers of abelian varieties}

\author[Thomas Agugliaro]{Thomas Agugliaro}
\email{thomas.agugliaro@math.unistra.fr}
\address{Institut de Recherche Mathématique Avancée, UMR 7501, Université de Strasbourg et CNRS, 7 rue René-Descartes, 67000 Strasbourg CEDEX, France}

\begin{abstract}
  We prove that the standard conjecture of Hodge type holds for powers of abelian threefolds. Along the way, we also prove the conjecture for powers of simple abelian variety of prime dimension over finite fields, and in other related cases based on the notion of Frobenius rank of Lenstra--Zarhin. The main tool is a result comparing two real fiber functors on tannakian categories. A second tool is a new an explicit description of simple Lefschetz motives over finite fields, in terms of ``enriched'' Frobenius eigenvalues.
\end{abstract}

\maketitle

\setcounter{tocdepth}{1}
\tableofcontents

\section{Introduction}

This paper is devoted to the study of the standard conjecture of Hodge type for powers of abelian varieties. The goal of this conjecture is to understand intersection forms on algebraic cycles, and more precisely the signature of such forms. The standard conjecture of Hodge type, as well as many other conjectures about algebraic cycles, is not stable under products, indeed there may be more algebraic cycles in a product of two varieties than in the varieties themselves. As an instance of our main result \Cref{thm:intro-frob-rank}, we prove the conjecture for powers of abelian threefolds. This is the first result on the conjecture where all powers of a complete family are settled.

The standard conjecture of Hodge type have been proposed by Grothendieck \cite{gro68} in $1968$, as a common generalization of the Hodge index theorem, the Hodge--Riemann bilinear relations and the positivity of the Rosati involution. The general conjecture is recalled in \Cref{conj:SCHT}, let us give here an instance of it. Let $X$ be a smooth projective geometrically irreducible variety of even dimension $d$ over a field $k$, and $\mathcal Z^n_\num(X)$ denote the finite dimensional $\bQ$-vector space of algebraic cycles modulo numerical equivalence, with $0 \leqslant n \leqslant d$.
\begin{conj}
  The intersection pairing
  \begin{align*}
    \mathcal Z^{d/2}_\num(X) \times  \mathcal Z^{d/2}_\num(X) &\to \bQ \\
    \alpha, \beta &\mapsto (-1)^{d/2} \alpha \cdot \beta
  \end{align*}
  has signature $(s_+, s_-)$ with
  \[ s_+ = \rho_{d/2} - \rho_{d/2 -1} + \rho_{d/2 - 2} - \dots + (-1)^{d/2}\rho_0, \]
  where $\rho_n = \dim(\mathcal Z^n_\num(X))$.
\end{conj}
Apart when $X$ is a surface or $X$ is a complex algebraic variety, cases which were already known when the conjecture was formulated, the only results on the conjecture are rather recent. In \cite{mil02}, Milne proved the conjecture for algebraic cycles which are linear combinations of intersections of divisors in abelian varieties (so called Lefschetz cycles). Ancona \cite{anc21} proved the conjecture for abelian fourfolds, introducing techniques from $p$-adic Hodge theory in the picture. His tools were used by Koshikawa \cite{kos24} to prove the conjecture for squares of simple abelian varieties of prime dimension over $\Fpbar$, as well as in \cite{agu24} by the author to prove the conjecture for infinitely many abelian varieties in each dimension. Using different techniques, Ito--Ito--Koshikawa \cite{IIK25} proved the conjecture for $S \times S$, where $S$ is any K3 surface, using the Kuga--Satake construction, and $\CM$-liftings to characteristic $0$.

\subsection{Main results}
The first result concerns simple abelian varieties of prime dimension.

\begin{theorem} \label{thm:intro-simple-prime}
  The standard conjecture of Hodge type holds for powers of simple abelian varieties over $\Fpbar$ of prime dimension.
\end{theorem}

An important property of the standard conjecture of Hodge type is its behavior under specialization: if a specialization of a variety $X$ satisfies the conjecture, so does $X$. However, being simple is not a property which can be tested on specializations, so that one cannot extend \Cref{thm:intro-simple-prime} to any field by a spreading out argument. Fortunately, we get the following corollary using a classification result on abelian threefolds \cite{zar15}.

\begin{cor} \label{cor:intro-threefolds}
  The standard conjecture of Hodge type holds for powers of abelien threefolds over any field.
\end{cor}

While studying the Tate conjecture for abelian varieties, Lenstra--Zarhin \cite{zar94} \cite{LZ93}, introduced an invariant of abelian varieties, called the Frobenius rank. Namely, given an abelian variety $A$ which is geometrically simple over $\bF_q$, with Frobenius eigenvalues $\pi_1, \dots, \pi_{2g}$, the Frobenius rank is the rank of the multiplicative group generated by $\frac{\pi_i^2}{q}$, $1\leqslant i \leqslant 2g$. This invariant measures the complexity of the ring of Tate classes in an abelian variety, c.f. \Cref{rmk:mult-relation-tate}. It has also been used by \cite{kos24} to deduce results on the standard conjecture of Hodge type from \cite{anc21}. The multiplicity of a simple abelian variety $A$ over $\bF_q$ is the multiplicity of the roots in its characteristic polynomial. The following result is the one from which \Cref{thm:intro-simple-prime} is deduced.

\begin{theorem} \label{thm:intro-frob-rank}
  Let $A$ be a geometrically simple abelian variety over $\bF_q$ of dimension $g$, Frobenius rank $r$ and multiplicity $m$. Assume that
  \begin{enumerate}
  \item $m$ is odd
  \item $r \geqslant \frac{g}{m} - 1$
  \item there exists a totally real field $D$ disjoint from the center $F$ of $\End(A)\otimes\bQ$ such that $A$ has $\CM$ by $D \cdot F$.
  \end{enumerate}
  Then all powers of $A$ satisfy the standard conjecture of Hodge type.
\end{theorem}

When $m = 1$, $D=\bQ$ will satisfy condition (3). As a consequence of \Cref{thm:intro-frob-rank} and the constructions of \cite{agu24}, we get

\begin{cor}
  For all prime numbers $p$ and all even integers $g \geqslant 4$, there exists infinitely many simple abelian varieties of dimension $g$ over $\Fpbar$ whose powers satisfy the standard conjecture of Hodge type, and such that there are Tate classes on $A$ which are not Lefschetz, and that do not come from specializing Hodge classes of $\CM$-liftings.
\end{cor}

In \cite{agu24}, we construct many abelian varieties satisfying the hypothesis of \Cref{thm:intro-frob-rank} with $m=1$. Using the same tools, it should be possible to construct abelian varieties satisfying the hypthesis for arbitrary $m > 1$.

There are two main difficulties in \Cref{thm:intro-frob-rank}. The first one is to prove the conjecture for \textit{powers} of $A$, which may contain a lot more algebraic cycles than $A$ itself, and even worse, it is not clear how the intersection forms of $A^n$ are related to intersection forms on $A$. Moreover, factors in the motive of the abelian varieties we consider have in general very high ranks, making the use of \cite[Theorem 8.1]{anc21} impossible. This was the key in \cite{anc21}, but also in \cite{agu24} and \cite{kos24}.

\subsection{Tools and idea of the proof of \Cref{thm:intro-frob-rank}}
A first step is to reformulate the standard conjecture of Hodge type using tannakian categories. Namely, in \Cref{prop:cyc-mot-comp} we write the intersection forms on algebraic cycles as realizations of quadratic forms defined on motives, or on Lefschetz motives (defined in \Cref{conv:lef-mot}).

Then, we define in \Cref{prop:arch-functor-lefschetz} a tensor functor
\[\Rinf: \Lmot(\bF_q) \to \Isoc_\bR,\]
on the category $\Lmot(\bF_q)$ of Lefschetz motives over the finite field $\bF_q$ taking values in the archimedean tannakian category $\Isoc_\bR$ of real isocristals. This functor can be interpreted as a functorial archimedean cohomology theory for abelian varieties, or an archimedean Tate module. The caveat of this realization is that it is only well-defined on Lefschetz motives, or in other words, it is a cohomology theory without the cycle class map. In absence of the cycle class map, it becomes unclear that such a functor is useful to study intersections of algebraic cycles. This motivates \Cref{thm:tann-main} and \Cref{cor:two-quad-forms}, which are among the main novelties of this paper. These results allows one to compare the previous archimedean realization and the spaces of algebraic cycles in a Lefschetz motive.

Then, a clean classification of simple Lefschetz motives over a finite field is given in \Cref{prop:lefmot-classification}. The simple Lefschetz motives are seen to be parametrized by \textit{enriched eigenvalues} of abelian varieties, which are defined in \Cref{def:enriched-eigenvalue-group}. This classification allows us to perform efficient dévissages in \Cref{sec:decomp-abel-motiv}. Namely, \Cref{prop:decomp-main} provides us with a decomposition of Lefschetz motives over finite fields. Then for orthogonality reasons, it suffices to prove a positivity result for each summand. By analyzing the decomposition, we see that the conjecture is already known for two of the factors, and \Cref{thm:exomot} gives a description of the other summands, which are called exotic.

From this description and \Cref{cor:two-quad-forms}, we are left to prove a positivity statement for a simple exotic Lefschetz motives $M$ of rank $2$ over a finite field with a quadratic form $q$. Such a motive $M$ does not lift to characteristic $0$, making the use of \cite[Theorem 8.1]{anc21} impossible. After an analysis of the situation, we see that $M^m$ does lift (where $m$ is the multiplicity). This allows us to use a theorem of Ancona--Marmora \cite[Theorem 3.2]{AM25}, which tells us that signature $(s_+, s_-)$ of an auxiliary quadratic form $q'$ on $M^m$ satisfies $4 \mid s_-$. On the other hand, if the quadratic form $q$ has signature $(2,0)$ (resp. $(1,1)$ or $(0,2)$), then $q^{\oplus m}$ on $M^m$ has signature $(2m, 0)$ (resp. $(m,m)$ or $(0,2m)$). Finally, we use \Cref{cor:two-quad-forms} to obtain that $q$ and $q'$ have the same signature, and as $m$ is odd, the only case where $4 \mid s_-$ is when $q$ is positive definite. This was the tannakian reformulation, and this concludes the proof.

\subsection{Organization}
This paper is organized as follows. In \Cref{sec:algebr-cycl-motiv}, we collect some basic properties on algebraic cycles, motives, tannakian categories and the standard conjecture of Hodge type. In \Cref{sec:abel-vari-abel}, we define the group of enriched eigenvalues associated to an abelian variety, and we use it to give a classification of simple Lefschetz motives over finite fields. Various corollaries of \Cref{thm:intro-frob-rank} are proved in \Cref{sec:positivity-results}. \Cref{sec:arch-fiber-funct} is devoted to the study of real valued fiber functors on tannakian categories. Namely, we define a functor from Lefschetz motives to real isocrystals, and we give results allowing one to compare quadratic forms in different realizations. Using the description of simple objects of the category of Lefschetz motives, in \Cref{sec:decomp-abel-motiv} we describe a fine decomposition result for motives of abelian varieties satisfying the hypthesis of \Cref{thm:intro-frob-rank}. Finally, in \Cref{sec:appl-posit-inters} we prove \Cref{thm:intro-frob-rank}, using the tannakian result \Cref{cor:two-quad-forms} on the category of Lefschetz motives, and the input of $p$-adic Hodge theory from \cite{AM25}.

\subsection*{Acknowledgements}
The author would like to thank G. Ancona for his many helpful comments during the writing of the article, E. Ambrosi, F. Charles, A. Marmora, J. Fres\'an and B. Kahn for useful discussions. The author is particularly grateful to E. Ambrosi for discussions about geometrically simple abelian varieties with no simple specializations.

This research was partly supported by the grant ANR-23-CE40-0011 of Agence National de la Recherche.

\section{Conventions} \label{sec:conventions}

\subsection{Number fields and $p$-adic fields}

A number field is a finite extension of $\bQ$. A totally real field $R$ is a number field, whose complex embedding are all real valued. A $\CM$-field is a number field $L$ which is a quadratic imaginary extension of a totally real field, and a $\CM$-algebra is a finite product of $\CM$-fields.

If $p$ is a prime number, a $p$-adic field is a finite extension $K$ of $\bQ_p$. The ring of integers of $K$ is denoted $\mathcal O_K$, it is the subring of elements of $K$ which are integral over $\Zp$. Moreover, $\mathcal O_K$ is a local field, and its residue field $\kappa$ is a finite field. We will also call $\kappa$ the residue field of $K$.

\subsection{Motives} \label{conv:motives}
We will work with the category of homological motives $\Mot(S)$ over a base $S$ with coefficients in $\bQ$. In practice, $S$ will be either a $p$-adic field $K$, its ring of integers $\mathcal O_K$, or a finite field $\bF_q$. We define it as in \cite[\S 2.3]{AM25}, as the quotient of the category of Chow motives over S with respect to homological equivalence. We will denote the unit object of $\Mot(S)$ by $\bUn$, and the Tate twists by $\bUn(n)$ for $n \in \bZ$.

Here, homological equivalence is taken with respect to any classical Weil cohomology (Betti, $\ell$-adic, de Rham). When $S$ is $p$-adic field or its ring of integers, it is independant of the choice thanks to the comparison theorems. When $S$ is a field of characteristic $p$, we will use homological equivalence with respect to $\ell$-adic cohomology with a fixed prime number $\ell \not = p$ in contexts in which it will plays no role.

Given a smooth projective scheme $X$ over $S$, we denote by $\frk h^\star(X)$ its associated motive. Whenever $X$ satisfies the standard Künneth conjecture, we will denote $\frk h^n(X)$ the summand of $\frk h^\star(X)$ whose realization is the $n$-th cohomology group of $X$.

We denote by $\Rz$ the functor
\begin{align*}
  \Rz: \Mot(k) &\to \Vect_\bQ \\
  M &\mapsto \Hom(\bUn, M)/\num.
\end{align*}
Recall that a map $f: \bUn \to M$ is numerically equivalent to $0$ if for any map $g:M \to \bUn$ we have $g\circ f = 0$.

We denote by $\Mot(S)^\ab$ the tannakian subcategory spanned by motives of abelian schemes. Let $K$ be a $p$-adic field, with ring of integers $\mathcal O_K$, residue field $k$ and maximal ideal $\frkp$. The specialization $\mathcal A \mapsto \mathcal A_\frkp$ where $A_\frkp$ is the special fibre of $\mathcal A$ at $\frkp$ extends to a functor
\[R_\frkp: \Mot^\ab(\mathcal O_K) \to \Mot(k). \]

\subsection{Realization functors} \label{conv:hdg-real}

Let $k$ be a field, with an algebraic closure $\overline k$, and $\ell \not = \mathrm{char}(k)$. The $\ell$-adic realization is the functor
\[ R_\ell: \Mot(k) \to \Rep_{\Ql}(\Gal(\overline k / k)),\]
such that $\h^\star(X)$ is sent to $H^\star(X \times_k \overline k, \Ql)$ endowed with the natural Galois action on $\ell$-adic cohomology.

When we have an embedding $\iota: k \hookrightarrow \bC$, we will consider the Hodge realization. The Hodge realization is the functor
\[R_{H,\iota}: \Mot(k) \to \HS_\bQ, \]
with values in $\bQ$-Hodge structures such that $\frk h^\star(X)$ is sent to $H^\star(X(\bC), \bQ)$ endowed with its Hodge structure. When $\iota$ is the inclusion from a subfield of $\bC$, we will omit it from the notations.

\subsection{Lefschetz motives} \label{conv:lef-mot}
We will also make use of the category of Lefschetz motives $\Lmot(S)$ over a base $S$ which is either a $p$-adic field $K$, its ring of integers $\mathcal O_K$, a finite field $\bF_q$ or $\Fpbar$. It was defined in \cite{mil99} when $S$ is a field, and is constructed in the same way as the usual category of motives, but allowing only abelian varieties, and correspondances which are sums of intersections of divisors, so called Lefschetz cycles. It also make sense when the base is $\mathcal O_K$ as follows:
\[ \Lmot(\mathcal O_K) \subset \Lmot(K)\]
is the smallest tannakian category containing $\frk h^1(A)$ for each abelian variety $A/K$ with good reduction.

We will denote by $\Lmot^{\mathrm{CM}}(K)$ the tannakian subcategory generated by motives of abelian varieties of $\CM$ type. Recall that an abelian variety $A$ is of $\CM$-type if it has dimension $g$ and that there exists a $\CM$-algebra $L$ of degree $2g$ with an embeding $L \hookrightarrow \End^0(A)$. When $A$ is an abelian variety over $K$, we will denote $\langle \frk h^1(A) \rangle^\otimes$ by $\langle A \rangle^\otimes$, with $\langle \frk h^1(A) \rangle^\otimes$ as in \Cref{def:subcat-generated-object}.

\subsection{Comparison and specializations} \label{conv:mot-lefmot-specializations}

Whenever one has a Lefschetz motive defined by a projector $P$ which is a Lefschetz cycle, one can see $P$ as an algebraic cycle, and it is still a projector in the usual category of motives. This procedure defines a functor
\[\Lmot(S) \to \Mot(S).\]
The specialization $A \mapsto A_s$, where $A_s$ is the special fibre of the Néron model of $A$ extends to a specialization functor
\begin{align*}
  R_\frkp: \Lmot(\mathcal O_K) &\to \Lmot(k) \\
   M &\mapsto M_{\frk p}
\end{align*}
where $k$ is the residue field of $\mathcal O_K$.
Moreover, we will also denote by $R_{H,\iota}$ (resp. $R_\ell$, $\Rz$) the composition $R_{H, \iota}$ from \Cref{conv:hdg-real} (resp. $R_\ell$, $\Rz$ from \Cref{conv:motives}) with the functor
\[ \Lmot(k) \to \Mot(k). \]

\section{Algebraic cycles and motives} \label{sec:algebr-cycl-motiv}
In this section, we will review algebraic cycles \eqref{sec:algebr-cycl-stand}, tannakian categories \eqref{sec:revi-tann-categ} and motives \eqref{sec:motiv-stand-conj}.

\subsection{Algebraic cycles and the standard conjecture of Hodge type} \label{sec:algebr-cycl-stand} In this section, we will fix a field $k$ and $H$ a Weil cohomology for smooth projective varieties over $k$, satisfying the Hard Lefschetz property. We will also consider $X/k$ a smooth projective, geometrically irreducible algebraic variety of dimension $d$, and $\mathcal L$ an ample divisor on $X$.

\begin{definition}
  For $0 \leqslant n \leqslant d$, we denote by $\CH^n(X)$ the Chow group of codimension $n$ cycles of $X$. That is the free $\bQ$-vector space generated by irreducible closed subvarieties of $X$, of codimension $n$, modulo rational equivalence. The intersection pairing
  \[ \CH^n(X) \times \CH^{d-n}(X) \to \CH^d(X) \xrightarrow{\deg} \bQ \]
  will be denoted $Z \cdot Z'$.
\end{definition}

Two cycles $Z_1, Z_2 \in \mathcal \CH^n(X)$ are said to be numerically equivalent if for all $Z' \in \CH^{d-n}(X)$, $Z_1 \cdot Z' = Z_2 \cdot Z'$. This defines an equivalence relation on $\CH^n(X)$, called numerical equivalence.

Two cycles $Z_1, Z_2$ are said to be equivalent with respect to the Weil cohomology $H$ if $\cl(Z_1) = \cl(Z_2)$. This defines an equivalence relation on $\CH^n(X)$, called homological equivalence.

\begin{definition}
  The quotient of $\CH^n(X)$ by numerical equivalence (resp. homological equivalence) will be denoted $\mathcal Z^n_{\num}(X)$ (resp. $\mathcal Z^n_{\eqhom}(X)$).
\end{definition}

Homological equivalence is finer than numerical equivalence, so that we get a map
\[ \mathcal Z^n_\eqhom(X) \to \mathcal Z^n_\num(X). \]
This map is compatible with the intersection products
\begin{align*}
  \mathcal Z^n_\num(X) &\times \mathcal Z^m_\num(X) \to \mathcal Z^{n+m}_\num(X), \\
  \mathcal Z^n_\eqhom(X) &\times \mathcal Z^m_\eqhom(X) \to \mathcal Z^{n+m}_\eqhom(X).
\end{align*}

\begin{theorem}[Hard Lefschetz] \label{thm:hard-lefschetz}
  Let $[\mathcal L] \in H^2(X)$ be the image of $\mathcal L$ by the cycle class map. The Lefschetz operator  $\cup [\mathcal L]$ induces isomorphisms for all $0 \leqslant n \leqslant d$
  \[ \cup [\mathcal L]^{d-n} : H^n(X) \to H^{2d-n}(X).\]
\end{theorem}
\begin{definition}
  The $n$-th primitive part of $X$ is the kernel $P^n_{\mathcal L}(X)$ of
  \[  \cup [\mathcal L]^{d-n+1}: H^n(X) \to H^{2d-n+2}(X).\]
\end{definition}
The following proposition is a consequence of \Cref{thm:hard-lefschetz}, it shows that the whole cohomology of $X$ can be expressed using primitive parts.
\begin{prop}
  For every $0 \leqslant n \leqslant d$, the following map is an isomorphism
  \[
    \begin{array}{rrcl}
      \bigoplus_{0 \leqslant i \leqslant \frac{n}{2}} P^{n-2i}_{\mathcal L}(X) & \ra & H^n(X) \\
      (\alpha_i)_{0 \leqslant i \leqslant \frac{n}{2}} & \mapsto & \sum_{0 \leqslant i \leqslant \frac{n}{2}} \alpha_i \cup [\mathcal L]^i
    \end{array}
  \]
  Moreover, the induced decomposition on $H^n(X)$ is orthogonal with respect to the intersection forms.
\end{prop}

We are now able to state the main positivity result in characteristic $0$, see \cite[Theorem 6.32]{voi02}.
\begin{theorem}[Hodge--Riemann relations] \label{thm:hdg-riem}
  Assume that $k = \bC$ and that $H$ is singular cohomology with rational coefficients, endowed with its pure Hodge structure. Then for all $0 \leqslant n \leqslant d$, the bilinear form
  \[\begin{array}{rrcl}
    \eta_{\mathcal L}: & P^n_{\mathcal L}(X) \times P^n_{\mathcal L}(X) & \to & H^{2d}(X, \bQ) \cong \bQ \\
    &(\alpha, \beta) & \mapsto & i^n [\mathcal L]^{d-n} \cup \alpha \cup \beta
  \end{array} \]
  is a polarization of Hodge structures.
\end{theorem}

The goal of the standard conjecture of Hodge type is to find a motivic version of \Cref{thm:hdg-riem}. The same definitions as before also make sense in the context of algebraic cycles.

\begin{definition} \label{def:primitive-cycles}
  The $n$-th primitive parts of $X$ is defined for $0 \leqslant n \leqslant d/2$ as the kernel $\mathcal P_{\mathcal L}^n(X)$ of
  \[ \cdot \mathcal L^{d-2n+1}: \mathcal Z^n_{\num}(X) \to \mathcal Z^{d-n+1}_\num(X). \]
\end{definition}
\begin{remark}
  The primitive part in cohomology is denoted with a straight $P$, whereas the letter denoting primitive part in algebraic cycles is calligraphed. Hence, the cycle class map relates $\mathcal P^n_{\mathcal L}(X)$ with $P^{2n}_{\mathcal L}(X)$.
\end{remark}

\begin{conj}[Standard conjecture of Hodge type] \label{conj:SCHT}
  For any $0 \leqslant n \leqslant d/2$, the pairing
  \[
    \begin{array}{rrcl}
      \eta_{\mathcal L}: &\mathcal P^n_{\mathcal L}(X) \times \mathcal P^n_{\mathcal L}(X) & \to & \bQ \\
                         &(\alpha, \beta) & \mapsto & (-1)^n \mathcal L^{d-2n} \cdot \alpha \cdot \beta
    \end{array} \]
  is positive definite.
\end{conj}

\begin{remark}
  If $k = \bC$, then \Cref{conj:SCHT} follows from \Cref{thm:hdg-riem}. Indeed, the cycle class map
  \[
    \begin{array}{rrcl}
      \cl^n: \mathcal Z^n(X) \to H^{2n}(X)
    \end{array}
  \]
  is compatible with intersection, and it has values in the $(n,n)$-part of the Hodge decomposition.
\end{remark}

\subsection{Review on tannakian categories:} \label{sec:revi-tann-categ}for a more in depth study of tannakian categories, see \cite{saa72} or \cite{del90}.
\begin{definition}
  Given a ring $K$, a $K$-linear tannakian category is a category $\mathcal T$
  \begin{itemize}
  \item endowed with the structure of a rigid symmetric monoidal category,
  \item which is an abelian $K$-linear category,
  \item such that there exists a faithfully flat map $K \to K'$ such that $\mcT$ admits a $K$-linear exact and faithful tensor functor to the category of projective $K'$-modules of finite rank. 
  \end{itemize}
\end{definition}

\begin{definition}
  Let $\mcT$ be a $K$-linear tannakian category, and $K \to K'$ be a faithfully flat map. A $K'$-linear fiber functor on $\mcT$ is a $K$-linear exact, faithful tensor functor
  \[\omega: \mcT \to \Vect_{K'},\] where $\Vect_{K'}$ denotes the category of projective $K'$-modules of finite rank.

  Given a fiber functor $\omega$ on $\mcT$ over $K$, we denote by $G_\omega$ the $K$-group $\underline{\Aut}^\otimes(\omega)$ of tensor automorphism of $\omega$.
\end{definition}

\begin{theorem}
  Given a $K$-linear tannakian category $\mcT$ with fiber functor $\omega$ over $K$, the functor
  \begin{align*}
    \mcT &\to \Rep_K(G_\omega) \\
    M &\mapsto \omega(M)
  \end{align*}
  is an equivalence of categories.
\end{theorem}

\begin{remark}
  Given a $K$-linear tannakian category $\mathcal T$ with a fiber functor $\omega_1$ over $K_1$. If $G_\omega$ is an abelian group scheme, then for any $\omega_2$ fiber functor over $K_2$, we have that $G_{\omega_2}$ is also an abelian group scheme. Moreover, in this case there exists an abelian group scheme $G/K$ such that for any fiber functor $\omega$ over $K'$, we have canonical identifications $G_\omega \cong G_{K'}$. More precisely, $G$ can be described as $\underline{\Aut}^{\otimes}(\id_{\mathcal T})$.
\end{remark}

\begin{definition}
  When $\mathcal T$ satisfies the conditions of the previous remark, we will say that $\mathcal T$ has an abelian tannaka group. Moreover, we will call $G=\underline{\Aut}^{\otimes}(\id_{\mathcal T})$ the Tannaka group of $\mathcal T$, even when $\mathcal T$ is non-neutral.
\end{definition}

\begin{definition} \label{def:subcat-generated-object}
  Given a tannakian category $\mathcal T$ and $V$ an object of $\mathcal T$, we denote by $\langle V \rangle^\otimes$ the smallest full tannakian subcategory of $\mathcal T$ containing $V$.
\end{definition}

\begin{remark}
  Alternatively, objects of $\langle V \rangle^\otimes$ are subquotients of $V^{\otimes n} \otimes (V^*)^{\otimes m}$ for $n$ and $m$ in $\bZ$.
\end{remark}

Using the monoidal structures of tannakian categories, we can do most of the usual multilinear algebra.

\begin{definition}
  Let $V$ be an object of a tannakian category $\mcT$, and $\mathbb L$ be an invertible object. A quadratic form on $V$, with values in $\mathbb L$ is a symmetric map
  \[ q: V \otimes V \to \mathbb L. \]
  More precisely, denote by $\tau$ the automorphism of $V \otimes V$ swapping the factors, we ask that the map $q$ satisfies $q \circ \tau = q$.

  A quadratic form is said to be non-degenerate if the map
  \[ V \to \underline{\Hom}(V, \mathbb L) \]
  obtained by adjunction is an isomorphism.
\end{definition}

Another classical multilinear object is the determinant, we describe here how it is obtained in tannakian categories.

\begin{definition}
  Assume that $K$ is a $\bQ$-algebra, let $V$ be an object of $\mathcal T$ of dimension $d$. Consider the action of $S_d$ on $V^{\otimes d}$ which permutes the pure tensors. We denote by $\bigwedge^dV$ the isotypical part of $V^{\otimes d}$ relative to the signature representation of $S_d$. More precisely the action of $S_d$ on $V^{\otimes d}$ extends to a module structure for the algebra $K[S_n]$, and
  \[\frac{1}{n!}\sum_{\sigma\in S_n} \varepsilon(\sigma)[\sigma] \in K[S_n],\]
  is an idempotent element. Hence it yields a projector of $V^{\otimes d}$, whose image is denoted $\bigwedge^d V$.
\end{definition}

\begin{lemma}
  The map
  \[K \to \Hom(\wedge^dV, \wedge^d V)\]
  mapping $1$ to $\id_{\bigwedge ^d V}$ is an isomorphism, whose inverse is the trace.
\end{lemma}
\begin{proof}
  As $\wedge^d V$ has dimension $1$, we have $\Tr(\id_{\wedge^d V}) = 1$, so that the trace is a left inverse of
  \[c: K \to \Hom(\wedge^dV, \wedge^d V).\]
  It suffices to check that $c$ is an isomorphism after tensoring by a faithfully flat $K$-algebra, so that we can assume that $\mcT$ has a $K$-linear fiber functor $\omega$. The composition
  \[ K \xrightarrow{c} \Hom(\wedge^d V, \wedge^d V) \xrightarrow{\omega} \Hom(\wedge^d \omega(V), \wedge^d \omega(V)) \]
  is an isomorphism. As $\omega$ is faithful, we deduce that $c$ is itself an isomorphism.
\end{proof}

\begin{definition} \label{def:det-charpoly}
  Let $V$ be an object of a $K$-linear tannakian category $\mcT$, and $u$ an endomorphism of $V$. Let $d$ denote the dimension of $V$, and $\bigwedge^d V$ be the $d$-th exterior power of $V$. Then
  \begin{itemize}
  \item $u$ induces an endomorphism $\bigwedge u \in \End(\bigwedge^d V)$. The scalar
    \[\det(u)=\Tr(\bigwedge u)\]
    is called determinant of $u$.
  \item Consider the object $V_{[X]}=V \otimes_K K[X]$, in the $K[X]$-linear tannakian category $\mcT_{[X]}= \mcT \otimes_K K[X]$, and the endomorphism $X \id_{V_{[X]}} - u$. The determinant of this endomorphism is an element of $K[X]$, which is denoted
    \[\chi_u(X) = \det(X\id_V -u).\]
  \end{itemize}
\end{definition}

The previous definition is intrinsic, it only uses the monoidal structure.

\begin{prop} \label{prop:charpoly-fiberfunctor}
  Given a tannakian category $\mcT$, and an endomorphism $u$ of an object $V$ of $\mcT$ we have for any fiber functor $\omega: \mathcal T \to \Vect_K$:
  \begin{enumerate}
  \item $\det(u) = \det(\omega(u))$
  \item $\chi_u = \chi_{\omega(u)}$
  \end{enumerate}
  where in both of these equalities the left hand side is defined by \Cref{def:det-charpoly}, and the right hand side coincides with the usual notion for linear endomorphisms.
\end{prop}

\begin{proof}
  Fiber functors commute with $\otimes$. As \Cref{def:det-charpoly} only depends on the monoidal structure, the determinant and characteristic polynomial will commute with $\omega$. Moreover in $\Vect_K$, the usual determinant can be computed as in \Cref{def:det-charpoly}.
\end{proof}

The following proposition will be useful to describe simple objects in tannakian categories.
\begin{prop} \label{prop:simple-object-tannaka-cat}
  Let $\mathcal T$ be a $F$-linear tannakian category, where $F$ is a field of characteristic $0$, and $L/F$ be a Galois extension with Galois group $G$. Assume that the Tannaka group of $\mathcal T$ is an algebraic torus $T$, and that $T$ is split by $L$. Then there is a bijection
  \[
    \begin{array}{rrcl} \label{eq:1}
    \{\text{Simple object of } \mathcal T \} &\cong &\mathrm{Orb}_G(X^\star(T_L)) \\
    M &\mapsto &\{\text{characters for the action of } T \text{ on } M\}.
    \end{array}
  \]
\end{prop}

\begin{remark} \label{rmk:quotient-sets-abelian-groups}
  In the previous proposition, it is tempting to write $X^\star(T_L)/G$ instead of $\mathrm{Orb}_G(X^\star(T_L))$, but one should be aware that the quotient is taken as sets, and not as abelian groups.
\end{remark}

\begin{proof}
  The category $\mathcal T_L$ has a split torus as Tannaka group. According to \cite[Chap 6, Prop 3.5.1]{saa72} there is a one-to-one correspondance between simple objects of $\mathcal T_L$ and $X^\star(T_L)$, sending an object $M$ to the unique character of $T$ acting non-trivially on $M$.

  First, let us prove that the map defined in the statement is well defined, that is that $G$ acts transitivelly on the characters of simple objects. Let $M$ be a simple object of $\mathcal T$, and let $S$ be the set of characters for the action of $T_L$ on $M$. Assume that we have a partition $S = S_1 \cup S_2$ which respects the action of $G$. Then $M_L = M_1 \oplus M_2$ where characters of $M_i$ are in $S_i$. By Galois descent of objects in Tannaka categories \cite[Chap 3, Prop. 1.2.3]{saa72}, we get a descent datum on $M_L$. Moreover, this descent datum stabilizes $M_1$ and $M_2$ (because there are no maps from $M_1$ to $M_2$ or from $M_2$ to $M_1$). Hence $M_1$ and $M_2$ come from objects of $\mathcal T$, and as $M$ is simple, either $S_1 = \emptyset$ or $S_2 = \emptyset$.

  Now, whenever one has two objects $M, N$ of $\mathcal T$ such that $\Hom(M, N) = 0$, then
  \[\Hom(M_L, N_L) = 0.\]
  It follows that the map \eqref{eq:1} is injective.

  For the surjectivity, we will again use Galois descent of objects in Tannaka categories. Consider $S$ an orbit for the action of $G$ on $X^\star(T_L)$. Fix $\chi \in S$, and consider
  \[ N = \bigoplus_{g \in G} L_\chi^g, \]
  where $L_\mu$ denotes the simple object of $\mathcal T_L$ associated to $\mu \in X^\star(\mathcal T)$. We endow $N$ with the descent datum
  \begin{align*}
    N^\sigma &\cong N \\
    (a_g)_{g \in G} &\mapsto (a_{\sigma^{-1} g})_{g \in G}
  \end{align*}
  under the identification of $(L_\chi^g)^\sigma$ with $L_\chi^{\sigma g}$. By Galois descent, there is an associated object $N'$ of $\mathcal T$ such that $N'_L \cong N$. In other words, $N'$ is the Weil restriction of scalars of $L_\chi$ along the extension $L/F$. Let $M$ be any simple component of $N'$, as $L_\chi^g \cong L_{g \chi}$, the set of characters for the action of $T_L$ on $M$ is $S$.
\end{proof}

\begin{remark}
  In the proof of the previous result, be aware that $M_L$ might not be isomorphic to $\bigoplus_{\mu \in S} L_{\mu}$. There might be some \textit{multiplicity} in the simple summands over $L$, for example this is the case for $\Isoc_\bR$ which we will describe now.
\end{remark}

\begin{definition} \label{def:real-isocrystals}
  We denote by $\Isoc_\bR$ the category whose objects are finite dimensional $\bC$-vector spaces $V$ endowed with
  \begin{enumerate}[label=(\roman*)]
  \item a decomposition $V = \bigoplus_{n \in \bZ} V_n$,
  \item a $\bC$-antilinear automorphism $\sigma: V \to V$,
  \end{enumerate}
  such that $\sigma$ respects the decomposition and for any $n$, $\sigma_{|V_n}^2 = (-1)^n \id_{V_n}$.

  Morphisms in $\Isoc_\bR$ are the $\bC$-linear maps which respect the structures. The $\Hom$ spaces are $\bR$-vector spaces, and we endow $\Isoc_{\bR}$ with the following monoidal structure:
  \[ (V \otimes W)_n = \bigoplus_{ a + b = n} V_a \otimes W_b. \]
  This makes $\Isoc_\bR$ into a $\bR$-linear tannakian category.
  
\end{definition}
\begin{remark}
  In $\Isoc_\bR$, the following object is simple:
  \[ V = \bC^2 = V_1, \]
  endowed with the automorphism $\sigma(z_1, z_2) = (-\overline{z_2}, \overline{z_1})$. Indeed, one can compute that $\End(V)$ is the quaternion algebra $\mathbb H$.

  But in the category $\Isoc_\bR \otimes_{\bR} \bC$, we have
  \[\End( V_\bC) = \End_{\Isoc_\bR}(V) \otimes \bC = \mathcal M_2(\bC).\]
  Hence there exist projectors exhibiting $V_\bC$ as the square of a simple object of $(\Isoc_\bR)_\bC$.
\end{remark}

\subsection{Motives and the standard conjecture of Hodge type} \label{sec:motiv-stand-conj}

The monoidal categories that we will be considering in this article are categories of motives. Informally, they are the categories you obtain when you transform the category of algebraic varieties into a monoidal category. For a precise discussion, see \cite{and04}. As explained in the conventions, we will be working with motives with respect to homological equivalence.

Let $\mathcal X$ be a smooth projective scheme of relative dimension $d$ over $S$ satisfying the Künneth standard conjecture, where $S$ is either a field or the ring of integers $\mathcal O_K$ of a $p$-adic field $K$. Let $\mathcal L$ be a relatively ample line bundle on $\mathcal X$.

\begin{definition}
  We say that $\mathcal X$ satisfies the relative standard conjecture of Lefschetz type if the Lefschetz operator induces an isomorphism for all $0 \leqslant n \leqslant d$
  \[\cup [\mathcal L]^{d-n}:\h^n(\mathcal X) \ra \h^{2d-n}(\mathcal X)(d-n). \]
  In any case, the $n$-th motivic primitive part is the kernel $\frk P^n_{\mathcal L}(X)$ of
  \[ \cup [\mathcal L]^{d-n+1}: \h^n(\mathcal X) \to \h^{2d-n+2}(\mathcal X)(d-n+1). \]
  The motivic intersection form $\eta_{\mathcal L, \mot}$ is the composite of the following maps
  \begin{align*}
    &\frk P^n_{\mathcal L}(\mathcal X) \otimes \frk P^n_{\mathcal L}(\mathcal X) \rightarrow \h^n(\mathcal X) \otimes \h^n(\mathcal X), \\
    &\h^n(\mathcal X) \otimes \h^n(\mathcal X) \cong \h^{n}(\mathcal X) \otimes \h^{2d-n}(\mathcal X)(d-n), \\
    &\h^{n}(\mathcal X) \otimes \h^{2d-n}(\mathcal X)(d-n) \ra \bUn(-n).
  \end{align*}
\end{definition}

\begin{prop}
  Assume that $\mathcal X$ satisfies the relative standard conjecture of Lefschetz type. Then for all $0 \leqslant n \leqslant d$, the map
  \[
    \begin{array}{rrcl}
      \bigoplus \cup [\mathcal L]^i: \bigoplus_{0 \leqslant i \leqslant \frac{n}{2}} \frk P^{n-2i}_{\mathcal L}(\mathcal X)(-i) & \ra & \h^n(\mathcal X)
    \end{array}
  \]
  is an isomorphism.
\end{prop}

\begin{remark}
  The functor $\Rz$ from \Cref{conv:motives} is lax--monoidal, so that whenever $\eta$ is a quadratic form on $M$, $\Rz(\eta)$ is a quadratic form on $\Rz(M)$.
\end{remark}

\begin{prop} \label{prop:cyc-mot-comp}
  Let $S=\Spec(k)$, $X$ a smooth projective variety over $k$, $\mathcal L$ is an ample divisor on $X$, $\mathcal P^n_{\mathcal L}(X)$ and $\eta_{\mathcal L}$ as in \Cref{def:primitive-cycles} and \Cref{conj:SCHT}. Then we have natural identifications
  \begin{align*}
    \Rz(\frk P^{2n}_{\mathcal L}(X)(n)) &= \mathcal P_{\mathcal L}^n(X) \\
    \Rz(\eta_{\mathcal L, \mot}(n)) &= \eta_{\mathcal L}.
  \end{align*}
\end{prop}

The following would be a generalization of \Cref{thm:hdg-riem} in a motivic setting.

\begin{conj} \label{conj:mot-SCHT}
  There is a notion of polarization on motives (c.f. \Cref{rem:mot-SCHT})such that
  \[\eta_{\mathcal L, \mot}: \frk P^n_{\mathcal L}(X) \otimes \frk P^n_{\mathcal L}(X) \to \bUn(-n)\]
  is a polarization for each $X$ and each ample divisor $\mathcal L$.
\end{conj}

Using \Cref{prop:cyc-mot-comp}, \Cref{conj:SCHT} would follow from \Cref{conj:mot-SCHT} in the same way as it follows from \Cref{thm:hdg-riem} over $\bC$.

\begin{remark} \label{rem:mot-SCHT}
  The theory of polarizations on tannakian categories has been axiomatized in \cite{saa72}. In this language, \Cref{conj:mot-SCHT} can be formulated by saying that there exists a polarization on \textit{the category} of motives, such that the bilinear forms $\eta_{\mathcal L, \mot}$ are positive relatively to this polarization. In \cite[Chap.6 \S 4.4]{saa72}, Saavedra shows that \Cref{conj:mot-SCHT} follows from \Cref{conj:SCHT}.
\end{remark}

The main positivity result that we will use in the paper is the following one \cite[Theorem 3.2]{AM25}.

\begin{theorem}[Ancona-Marmora] \label{thm:ancona-marmora}
  Let $K$ be a $p$-adic field, with ring of integers $\mathcal O_K$, and fix an embedding $\sigma: K \to \bC$. Let $M$ be a motive in $\Mot(\mathcal O_K)$ pure of weight $2n$, with $\CM$ by $F$ and $\eta$ be a quadratic form on $M(n)$, with values in $\bUn$. Assume that
  \begin{enumerate}
  \item $\dim \Rz(M_{\frk p}(n)) = \dim M$,
  \item For every $z \in F$, the adjoint of $z$ with respect to $\eta$ is $\overline z$.
  \item $R_B^\sigma(\eta)$ is a polarization of the Hodge structure $R_B^\sigma(M)$,
    
  \end{enumerate}
  Then the quadratic form $\Rz(\eta_\frkp)$ on $\Rz(M_{\frk p}(n))$ has signature $(s_+, s_-)$ with $4 | s_-$.
\end{theorem}

\begin{remark}
  The conclusion of the theorem is about algebraic cycles on the special fibre of the motive, that is algebraic cycles in characteristic $p$. But the assumptions are Hodge theoretic, they are about the cohomology of a complex variety. The bridge between these setups is made by techniques of $p$-adic Hodge theory.
\end{remark}

\begin{remark} \label{rem:cycle-eigen}
  Let $k$ be a finite field of characteristic $p$. Let $\ell \not = p$ be a prime number, and $\Mot(k)$ be the category of motives with respect to $\ell$-adic homological equivalence. The $\ell$-adic realization (see \Cref{conv:hdg-real}) is a faithful functor
  \[ R_\ell: \Mot(k)_{\Ql} \to \Rep_{\Ql}(\Gal(\overline k/k)). \]
  In particular, for any motive $M$, there is an injective map
  \[\Hom(\bUn, M) \otimes \Ql \hookrightarrow R_\ell(M)^{\Frob_k}. \]
  Hence, if $1$ is not an eigenvalue of Frobenius on $M$, then $\Hom(\bUn, M) = 0$. The converse is expected to be true, as it is a consequence of the following conjecture.
\end{remark}

\begin{conj}[Tate] \label{conj:tate}
  Let $X$ be a smooth projective variety over a finite field $k$. For all $n$, the cycle class map
  \[ \cl^n: \mathcal Z^n_\eqhom(X) \to H^{2n}_\ell(X)(n)^{\Frob_k}\]
  is surjective.
\end{conj}

\section{Enriched eigenvalues and Lefschetz motives} \label{sec:abel-vari-abel}

In this section, we define a new invariant of abelian varieties, which we call the group of enriched eigenvalues. Then, we use this invariant to give an explicit description of simple Lefschetz motives in \Cref{prop:lefmot-classification}. Then we describe the dimension of a Lefschetz motive in terms of its associated enriched eigenvalues in \Cref{prop:motive-rank}. Finally, we finish the section with two lemmas which are used in \Cref{sec:appl-posit-inters}. First, let us recall the standard definitions.
\subsection{Enriched eigenvalues of abelian varieties} \label{sec:enrich-eigenv-abel}
Let $p$ be a prime number. Let $A$ be an abelian variety of dimension $g$ over a finite field $\bF_q$ of characteristic $p$. Consider
\[\End^0(A)=\End(A) \otimes_{\bZ} \bQ\]
and $F$ the center of $\End^0(A)$. The geometric Frobenius $\pi$ of $A$ is a central endomorphism of $A$. Moreover by \cite[Theorem 2]{tat66}, $F= \bQ(\pi)$.

\begin{definition}
  When $A$ is simple, the multiplicity of $A$ is the integer $m$ such that
  \[m^2=\dim_F\End^0(A). \]
\end{definition}

\begin{remark}
  By \cite[Theorem 1]{tat66}, $m$ is also the multiplicity of the roots of the characteristic polynomial of $\pi$, that is it satisfies $2g = m [F:\bQ]$.
\end{remark}

Let us go back to the case where $A$ is not necessarily simple. Here $F$ is a semisimple commutative algebra, finite over $\bQ$, hence it is a product of number fields. Let $\tilde F$ denote a Galois closure of a compositum of the factors of $F$, and $G = \Gal(\tilde F/\bQ)$.
The following notion of rank appears in \cite{LZ93}, \cite{zar94}, and is useful to study Tate classes in powers of abelian varieties.

\begin{definition}
  Denote by $\pi_1, \dots, \pi_d$ the images of $\pi$ in $\tilde F$ along maps $F \to \tilde F$. Consider the subgroup $\Gamma$ of $\tilde F^\times$ generated by $\pi_1, \dots, \pi_d$. The Frobenius rank of $A$ is
  \[ \rk(A) := \rk (\Gamma) -1 = \rk(\Gamma/\langle p \rangle).\]
\end{definition}

\begin{remark}
  If $A$ is simple with multiplicity $m$, then $d = \frac{2g}{m}$, and as $\pi_i  \overline \pi_i = q = 1$ in $\tilde F^\times/\langle p \rangle$, we have
  \[ 0 \leqslant \rk(A) \leqslant \frac{g}{m}.\]
\end{remark}

\begin{notation} \label{not:ab-var}
  Let $A$ be an abelian variety over $\bF_q$. We will denote by $F$ the center of $\End^0(A)$ ($F$ is a product of number fields), by $\tilde F$ a Galois closure of a compositum of the factors of $F$, and $r$ the Frobenius rank of $A$. Moreover, when $A$ is simple, we will denote by $m$ the multiplicity of $A$, which is the multiplicity of the roots of the characteristic polynomial of $A$.
\end{notation}

The following definition introduces the new invariant of this section.

\begin{definition} \label{def:enriched-eigenvalue-group}
  Let $A$ be an abelian variety over $\bF_q$, with Frobenius endomorphism $\pi \in F$, and let
  \[\Pi = \{ \sigma(\pi) \mid \sigma: F \to \tilde F \}.\]
  The group $\Eig(A)$ of enriched eigenvalues of $A$ is the free abelian group generated by $\Pi \cup q$, denoted multiplicatively, modulo the relations $[\pi][\overline \pi]=[q]$ with $\pi \in \Pi$.
\end{definition}

\begin{definition} \label{def:realization-eigenvalues}
  The realization map is the map
  \[\rho:\Eig(A) \to \tilde F^\times \]
  induced by the inclusion $\Pi \cup q \subset \tilde F^\times$.
\end{definition}

\begin{remark} \label{rmk:mult-relation-tate}
  By \cite[Theorem 3.4.3]{zar94}, if $\rho$ is injective, then all Tate classes in the cohomology of powers of $A$ are Lefschetz classes. So that $\Ker(\rho)$ measures the defect of Tate classes to be Lefschetz classes.

  We will call $\Ker(\rho)$ the group of nontrivial multiplicative relations between $\pi_1, \dots, \pi_d$, it is a free abelian group of rank $\frac{g}{m} - \rk(A)$.
\end{remark}

\subsection{Simple Lefschetz motives}

In this article, algebraic cycles come in two flavors:
\begin{enumerate}[label=(\roman*)]
\item the cycles for which we have to prove positivity in \Cref{conj:SCHT}
\item the cycles that we use to manipulate and decompose motives.
\end{enumerate}
The point of using Lefschetz motives is to have a formalism that distinguishes between these two kinds of algebraic cycles. For more precisions, see \ref{conv:lef-mot}.

\begin{remark}
  As homological equivalence coincides with numerical equivalence for Lefschetz cycles \cite[Prop 5.2]{mil99a}, we have the usual realization functors which are well defined on Lefschetz motives. Hence, it makes the category of Lefschetz motive into a tannakian category. Moreover by \cite{jan92}, the category of Lefschetz motive is semisimple.
\end{remark}

\begin{prop} \label{prop:lefgroupchar}
  Let $A$ be an abelian variety over $\bF_q$, whose geometric endomorphisms are defined over $\bF_q$, and $L_A$ be the Tannaka group of $\langle A \rangle^\otimes$. Then $L_A$ is an algebraic tori split by $\tilde F$ (c.f. \Cref{not:ab-var}), and there is a natural Galois-equivariant isomorphism of abelian groups
  \[ X^\star(L_{A,\tilde F}) \cong \Eig(A).\]
\end{prop}

\begin{proof}
  Let $A$ be an abelian variety over $\bF_q$, with Tannaka group $L_A$. Assume that the geometric endomorphisms of $A$ are defined over $\bF_q$. Then the Lefschetz group of $A$ is the same as the Lefschetz group of $A_{\Fpbar}$.

  Let $\Pi$ be the set of eigenvalues of $A$ in a spliting field $\tilde F$ of the characteristic polynomial. By \cite[Lemma 4.2]{mil99}, we have an identification
  \[X^*(L_{A, \tilde F}) = \{f: \Pi \to \bZ \}/ \{f \: | \; f = \iota f \text{ and } \sum_{\pi \in \Pi} f(\pi) = 0 \},\]
  where $\iota$ denote complex conjugation. In particular, $L_A$ is split by $\tilde L$.
  Consider the Galois equivariant map 
  \[\Pi \cup \{q\} \to X^\star(L_{A, \tilde F}) \]
  mapping $\pi \in \Pi$ to the characteristic function $f_\pi$ of $\pi$ and $q$ to $f_\pi + f_{\overline \pi}$ for any $\pi$ (in the quotient, this is independant of $\pi \in \Pi$).
  It induces a unique map,
  \[ \Eig(A) \to X^\star(L_{A, \tilde F}), \]
  by the universal propery of the free group with relations. This map is Galois equivariant, and surjective because any function $f: \Pi \to \bZ$ can be written as a sum of characteristic functions of elements. For injectivity, let
  \[ \lambda = [q]^a \cdot \prod_{\pi \in \Pi} [\pi]^{a_\pi}\]
  be in the kernel. By replacing $[q]$ by $[\pi_0] [\overline{\pi_0}]$, we can assume that $a=0$. Then, the map
  \[ f: \pi \mapsto a_\pi \]
  satisfies $f = \iota f$ and $\sum_{\pi \in \Pi} f(\pi) = 0$. In particular, for all $\pi$, $a_\pi = a_{\overline \pi}$, so that we can write
  \begin{align*}
    \lambda &= \prod_{\{\pi, \overline \pi\} \in \Pi/\iota} ([\pi][\overline \pi])^{a_\pi} \\
            &= \prod_{\{\pi, \overline \pi\} \in \Pi/\iota} [q]^{a_{\pi}}, \\
            &= [q]^{\frac{1}{2}\sum_{\pi \in \Pi} a_{\pi}} \\
            &= 1.
  \end{align*}
\end{proof}

\begin{definition} \label{def:enriched-eigenvalues}
  Let $A$ be an abelian variety over $\bF_q$, whose geometric endomorphisms are defined over $\bF_q$, and $M$ be a Lefschetz motive of rank $r$ in $\langle A \rangle^\otimes$. Let $\chi_1, \dots, \chi_r$ be the characters for the action of $L_A$ on $M$. We will call enriched eigenvalue of $M$, and denote $\lambda_1, \dots, \lambda_r$ the elements of $\Eig(A)$ corresponding to $\chi_1, \dots, \chi_r$ via the identification of \Cref{prop:lefgroupchar}.
\end{definition}

\begin{prop} \label{prop:realization-enriched-classical}
  Let $A$ be an abelian variety over $\bF_q$, whose endomorphisms are defined over $\bF_q$, and $M \in \langle A \rangle^\otimes$ be a Lefschetz motive with enriched eigenvalues $\lambda_1, \dots, \lambda_r \in \Eig(A)$. Then the action of Frobenius on $R_\ell(M)$ has eigenvalues $\rho(\lambda_1), \dots, \rho(\lambda_r)$ (where $\rho$ is defined in \Cref{def:realization-eigenvalues}).
\end{prop}

\begin{remark}
  \Cref{prop:realization-enriched-classical} explains why we chose the name \textit{enriched} eigenvalues: one can recover the usual eigenvalues from the enriched eigenvalues.
\end{remark}

\begin{proof}
  Let $F$ be the center of $\End^0(A)$. It is a product of $\CM$ fields and totally real fields, denote by $\tilde F$ a Galois closure of a compositum of these fields. Let $M \in \langle A \rangle^\otimes$, and denote $M_{\tilde F} \in \langle A \rangle^\otimes_{\tilde F}$ the object obtained by extension of scalars. By \Cref{prop:lefgroupchar}, the Lefschetz group $L_A$ is split by $\tilde L$. Hence we can decompose $M_{\tilde F}$ as a sum of its isotypical components
  \[ M_{\tilde F} = M_{\chi_1} \oplus \dots \oplus M_{\chi_r} \]
  where $\chi_1, \dots, \chi_r$ are the characters for the action of $L_A$ on $M$.
  Let $\ell \not = p$, and fix an embedding $\tilde F \subset \overline{\bQ}_\ell$. Now, we consider the $\ell$-adic realization
  \[R_\ell(M_{\tilde F}) = R_\ell(M_{\chi_1}) \oplus \dots \oplus R_\ell(M_{\chi_r}). \]
  The eigenvalues of Frobenius on $R_\ell(M)$ are the same as the Frobenius eigenvalues on $R_\ell(M_{\tilde F})$, so that it suffices to show that the Frobenius eigenvalues on $R_\ell(M_{\chi_i})$ is $\rho(\lambda_i)$, where $\lambda_i$ is associated to $\chi_i$ via \Cref{prop:lefgroupchar}.

  In turn, it suffice to consider the simple Lefschetz motive $N_\chi$ in $\langle A \rangle^\otimes_{\tilde L}$ of type $\chi \in X^\star(L_{A, \tilde F})$ (which are defined in \cite[Chap 6, Prop 3.5.1]{saa72}), and show that the eigenvalue on Frobenius on $R_\ell(N_\chi)$ is $\rho(\lambda)$, where $\lambda$ is associated to $\chi$ via \Cref{prop:lefgroupchar}. The character associated to $N_\chi \otimes N_{\chi'}$ is $\chi + \chi'$, and the associated eigenvalue of Frobenius is the product of the eigenvalues on $N_\chi$ and on $N_{\chi'}$. Hence it suffices to show the desired equality on a generating set of characters.

  Such a set is given by the characters of $L_A$ acting on $\frk h^1(A)$. According to \cite[Remark 1.10]{mil99}, we have that
  \[L_A(\bQ) = \{ \alpha \in F \mid \alpha \cdot \overline \alpha \in \bQ^\times \}, \]
  and moreover the Frobenius $\pi \in F$ corresponds to an element of $L_A(\bQ)$. This equality is such that $\pi$ acts on $R_\ell(\frk h^1(A))$ as the geometric Frobenius acting by functoriality. Hence its eigenvalues are the eigenvalues of Frobenius. The other representation theoretic description of these eigenvalues is the following: they are the $\chi(\pi)$ where $\chi$ runs through the characters for the action of $L_A$ on $\frk h^1(A)$. These characters are just the restrictions of ring maps $F \to \tilde F$, so that the $\chi(\pi) \in \tilde L$ is $\pi_i$ for some $i$, which is the realization of $[\pi_i]$.
\end{proof}

Tate conjecture \Cref{conj:tate} predicts that there should be a nonzero map between two motives having an eigenvalue in common. The following proposition, which holds inconditionnaly, gives an analog of this with enriched eigenvalues. The main component of its proof is Tate's theorem on divisors of abelian varieties \cite{tat66}.

\begin{prop} \label{prop:lefmot-classification}
  Let $A/\bF_q$ be an abelian variety whose geometric endomorphisms are defined over $\bF_q$. Then the map
  \[
  \begin{array}{rrcl}
    \{\text{Simple objects of } \langle A \rangle^\otimes \} & \to & \mathrm{Orb}_G(\Eig(A)) \\
    M \quad & \mapsto & \{\text{enriched eigenvalues of } M \}
  \end{array}
  \]
  is bijective.
\end{prop}
\begin{remark}
  As in \Cref{rmk:quotient-sets-abelian-groups}, $\mathrm{Orb}_G$ denotes the set of orbits, and not the quotient as abelian groups.
\end{remark}
\begin{proof}
  By \Cref{prop:lefgroupchar}, $\langle A \rangle_{\tilde F}$ has a split torus $L_{A, \tilde L}$ as Tannaka group.
  Hence, by \Cref{prop:simple-object-tannaka-cat}, we get a bijection
  \[
    \begin{array}{rrcl}
      \{\text{Simple objects of } \langle A \rangle^\otimes \} & \cong & \mathrm{Orb}_G(X^\star(L_A)) \\
      M & \mapsto &\{\text{characters for the action of } L_{A, \tilde L} \text{ on } M \}.
    \end{array}
  \]
  We conclude by using the isomorphism
  \[ \mathrm{Orb}_G(X^\star(L_A))\cong \mathrm{Orb}_G(\Eig(A)) \]
  provided by \Cref{prop:lefgroupchar}.
\end{proof}

\begin{prop} \label{prop:motive-rank}
  Let $A$ be an abelian variety over a finite field $\bF_q$, whose geometric endomorphisms are defined over $\bF_q$, and $M$ be a simple Lefschetz motive in $\langle A \rangle^\otimes$ with $S$ as set of enriched eigenvalues. Assume that $M$ is pure of weight $2n$, and that the Frobenius acts on $M$ by multiplication by $q^n$, then $\dim(M) = \# S$.
\end{prop}
\begin{proof}
  Let $Z$ denote the center of $\End(M)$. Then $\End(M)$ being a central simple $Z$-algebra, it has dimension $a^2$ as a $Z$-vector space. The integer $a$ is called multiplicity of $M$, and it is also the order of $[\End(M)]$ in $\mathrm{Br}(Z)$. With these notations, we have $\dim(M) = a \deg(Z)$. Indeed, consider $L$ a maximal commutative subalgebra of $\End(M)$, and $\tilde L$ a Galois closure of $L$. Then $L$ is an extension of $Z$ of degree $a$, and $\deg(L) = a \deg(Z)$. Moreover, one has $\deg(L) = \dim(M)$, because $M_{\tilde L}$ can be decomposed into a direct sum of characters indexed by embeddings $L \to \tilde L$.

  In addition, $Z$ can be computed as the fixed field of $\mathrm{Fix}_G(s)$ for any $s\in S$, where
  \[G=\Gal(\tilde F/\bQ).\]
  Hence $\deg(Z) = \#S$. To conclude, it suffices to prove that $a=1$. To do so, we can compute the local invariants of $\End(M)$ as a central simple $Z$-algebra.

Consider the $\bQ$-linear tannakian category $\langle M(n) \rangle$. For all $\ell \not = p$, $\ell$-adic cohomology defines a fiber functor on $\langle M(n) \rangle$. From this, we deduce that the local invariants at $\ell \not = p$ vanish. Moreover, as the Frobenius acts on $M$ by multiplication by $q^n$, it will act on $M(n)$ by the identity. Hence the $F$-isocrystal associated to $M(n)$ has slope $0$. The functor on Lefschetz motives extending

\[A \mapsto H_{\mathrm{crys}}(A\times_{\bF_q} \Fpbar/ W(\Fpbar)) \]
maps $\langle M(n) \rangle$ to $F$-isocrystals of slope $0$ over $\Fpbar$. This category being equivalent to the category of $\bQ_p$-vector spaces, crystalline cohomology defines a $\Qp$-linear fiber functor on $\langle M(n) \rangle$, so that the invariants at $p$ also vanish. As $M(n)$ has weight $0$, and using \Cref{cor:rinf-fiber-fun}, $\Rinf$ defines a $\bR$-linear fiber functor, so that the invariants at the infinite place vanish.
\end{proof}

We end this section with two lemmas on simple Lefschetz motives which will be used in \Cref{sec:appl-posit-inters}. They concern $\CM$ Lefschetz motives as defined in \Cref{conv:lef-mot}.

\begin{lemma} \label{lem:motive-polarization}
  Let $k$ be a subfield of $\bC$, and $N\in \Lmot^\CM(k)$ be a geometrically simple object of even weight which is not a Tate twist of the unit. Then $N$ has complex multiplication by the field $L = \End(N)$, and there exists a quadratic form $\eta$ on $N$ such that
  \begin{enumerate}[label=\roman*)]
  \item for every $z \in L$, the adjoint of $z$ with respect to $\eta$ is $\overline z$.
  \item $R_H(\eta)$ is a polarization of Hodge structure.
  \end{enumerate}
\end{lemma}

\begin{proof}
  By definition, $N$ is a direct summand of the motive $\h^\star(A)$ associated to a $\CM$ abelian variety $A$. Let $q$ be a quadratic form on $\h^\star(A)$ such that $R_H(q)$ is a polarization of Hodge structure. Then the restriction $\eta$ of $q$ to $N$ is such that $R_H(\eta)$ is a polarization on the Hodge structure $R_H(N)$, in particular $\eta$ is nondegenerate.

  Adjunction with respect to $\eta$ induces an involution $x \mapsto x^{\dagger}$ on the commutative field $L = \End(N)$. As $R_H(\eta)$ is a polarization, we get that $x \mapsto x^\dagger$ is a positive involution, so that $L$ is either a totally real field or a $\CM$ field with complex conjugation $\dagger$. We conclude by \Cref{lem:totally-real-endo}, because $N$ is not a Tate twist of the unit.
\end{proof}

\begin{lemma} \label{lem:totally-real-endo}
  Let $k$ be a subfield of $\bC$, and $N$ be a geometrically simple Lefschetz motive in $\Lmot^\CM(k)$. Then either $N$ is a Tate twist of the unit, or the field $\End(N)$ is not totally real.
\end{lemma}

\begin{proof}
  Let $A$ be an abelian variety over $k$ of $\CM$ type such that $N$ is a submotive of $\h^1(A)^{\otimes n}$ for some $n$. We can assume that $k$ is algebraically closed, and that $A = \prod_i A_i$ has $\CM$ by a $\CM$-algebra $L=\prod L_i$, with $\CM$ type $\Phi \subset \Hom(L, k)$. Let $\tilde L$ be the Galois closure in $k$ of a compositum of the $L_i$'s. Then, we can identify $\Hom(L, k)$ with $\Hom(L, \tilde L)$, and we will denote $G = \Gal(\tilde L/\bQ)$ and $\iota \in G$ the complex conjugation.
  
  In \cite[Proposition 2.5]{mil99}, Milne considers Galois orbits of $\CM$ types. But according to \cite[Proposition 2.2]{mil99}, we have that $\Hom(L, \tilde L) \cong \mathcal C$ where $\mathcal C$ is the Galois orbit of $\Phi$. We deduce from \Cref{prop:simple-object-tannaka-cat} that $N$ corresponds to a conjugacy class $C$ of elements in
  \[ X^\star(L(A)) = \frac{\{f: \Hom(L, \tilde L) \to \bZ\}}{ \{f \; | \; f = \iota f \text{ and } \sum_{\sigma \in \Hom(L,\tilde L)} f(\sigma) = 0 \}}\]
  for the action of $G = \Gal(\tilde L/\bQ)$. Moreover, the endomorphism algebra of $N$ is isomorphic to the fixed field of $\Fix_G([f])$ for any $[f] \in C$. Hence, $N$ is real if and only if $\iota \in \Fix_G([f])$. Let $f$ be a lift of $[f]$ as a function $f: \Hom(L, \tilde L) \to \bZ$. If $\iota [f] = [f]$, then there exists $g$ with $\iota g = g$ and $\iota f = f + g$. By applying $\iota$ again, we see that $g = 0$, so that $\iota f = f$.

  It follows that $[f]$ corresponds to a Tate twist, because the element $\delta_{\sigma} + \iota \delta_{\sigma} \in X^\star(L(A))$ is independant of $\sigma: L \to \tilde L$, and corresponds to the simple object $\bUn(-1)$.
  \end{proof}

\section{Main results} \label{sec:positivity-results}

In this section, we state the main result of the paper, \Cref{thm:main} (i.e. \Cref{thm:intro-frob-rank} from the introduction), whose proof is given in \Cref{sec:appl-posit-inters}. We then deduce from it the theorems announced in the introduction.

Let $A$ be a geometrically simple abelian variety over $\bF_q$, and consider its Frobenius rank $r$, its multiplicity $m$ and its center $F$ as in \Cref{not:ab-var}.
\begin{remark} \label{rem:mult-rank-fpbar}
  By considering models, the notions of Frobenius rank and multiplicity also make sense for a simple abelian variety over $\Fpbar$.
\end{remark}

\begin{theorem}\label{thm:main}
  Let $A$ be a simple abelian variety over $\Fpbar$, of dimension $g$, Frobenius rank $r$ and multiplicity $m$ (c.f. \Cref{rem:mult-rank-fpbar}). Assume that
  \begin{enumerate}
  \item $m$ is odd,
  \item $r \geqslant \frac{g}{m}-1,$
  \item there exists a totally real field $D$ disjoint from $F$ such that $A$ has $\CM$ by $D \cdot F$, i.e. that $\deg D = m$ and there is a map $D \cdot F \to \End^0(A)$.
  \end{enumerate}
  Then all powers of $A$ satisfy the standard conjecture of Hodge type.
\end{theorem}
We will give a proof in \Cref{sec:appl-posit-inters}.

\begin{remark} \label{rem:main-m}
  When $m = 1$, assumption (3) is automatic: $D = \bQ$ satisfies the condition. In \cite{agu24}, we construct many abelian varieties satisfying the hypothesis of the \Cref{thm:main} with $m=1$. It should be possible to adapt the techniques to construct abelian varieties over $\Fpbar$ satisfying the hypothesis for arbitrary $m>1$.
\end{remark}

\begin{cor} \label{cor:conj-prime-dimension}
  Let $A$ be a simple abelian variety over $\Fpbar$, of dimension $g$. Assume that $g$ is a prime number, then all powers of $A$ satisfy the standard conjecture of Hodge type.
\end{cor}

\begin{proof}
  Let $A$ be a simple abelian variety over $\Fpbar$ of prime dimension $g$. Let $m$ be the multiplicity of $A$, and $r$ be the Frobenius rank of $A$. Then by a theorem of Tankeev \cite{tan83}, we have that $r \geqslant g-1$. As $r \leqslant \frac{g}{m}$, we get that $m=1$ which falls in the case described by \Cref{rem:main-m}. Hence we conclude by \Cref{thm:main}.
\end{proof}

\begin{cor}
  The standard conjecture of Hodge type is satisfied by power of abelian threefolds.
\end{cor}
\begin{proof}
  By a specialization argument \cite[Prop 3.16]{anc21}, it suffices to prove the result for abelian threefolds over $\Fpbar$. Let $A$ be an abelian variety over $\Fpbar$, and let us prove the conjecture for powers of $A$. If $A$ is simple then the conjecture follows from \Cref{cor:conj-prime-dimension}. Otherwise, it follows from \cite[Thm 1.1]{zar15} that $A$ is neat, i.e. that Tate classes in the powers of $A$ are Lefschetz classes. Hence the standard conjecture of Hodge type for powers of $A$ follows from \cite[Remark 3.7]{mil02} or \cite[Proposition 5.4]{anc21}.
\end{proof}

\section{Archimedean fiber functors} \label{sec:arch-fiber-funct}

A classical argument of Serre shows that there is no cohomology theory for varieties over finite fields which takes values in real vector spaces. But there is no known obstruction for having a cohomology theory taking values in $\Isoc_\bR$ as defined in \Cref{def:real-isocrystals}. The goal of this section is to exploit the existence of an archimedean fiber functor over the category of Lefschetz motives. It plays the role of a cohomology theory for abelian varieties, but without a cycle class map.

\begin{definition}
  For each $n \in \bZ$, the object $\bR(n)$ of $\Isoc_\bR$ will be the graded vector space $\bC$ in degree $-2n$, with the complex conjugation as antilinear automorphism. For simplicity, we will denote $\bR(0)$ by $\bR$.
\end{definition}

\begin{remark}
  The object denoted $\bR$ in the previous definition is the unit object of $\Isoc_\bR$.
\end{remark}

\begin{definition} \label{def:positive-definite}
  Let $(V, \sigma)$ be an object of $\Isoc_\bR$, pure of weight $n$, and
  \[ \eta: V \otimes V \to \bR(n) \]
  be symmetric if $n$ is even, and antisymmetric if $n$ is odd. We say that $\eta$ is positive definite if for each $v \in V \setminus \{0\}$,
  \[\eta(v, \sigma v) \in \bR_{>0}.\]
  In other words, we say that $\eta$ is positive definite if the Hermitian form
  \[ \hat \eta: (v,w) \mapsto \eta(v, \sigma w),\]
  is positive definite.
\end{definition}

\begin{prop} \label{prop:fiber-even-isoc}
  Let $\Isoc_{\bR}^{\text{even}}$ be the category of real isocrystals of even weight. The functor mapping an even weight isocrystal $(V, \sigma)$ to the set of fixed points $V^\sigma=\{v \in V \mid \sigma(v) = v\}$ is a fiber functor
  \[\omega: \Isoc_{\bR}^{\text{even}} \to \Vect_{\bR}. \]
  Moreover, for any quadratic form $\eta$ on $V \in \Isoc_{\bR}^{\text{even}}$, $\eta$ is positive definite in the sense of \Cref{def:positive-definite} if and only if the quadratic form $\omega(\eta)$ is positive definite on the $\bR$-vector space $\omega(V)$.
\end{prop}

\begin{proof}
  For for an object $(V, \sigma)$ of even weight, we have $\sigma^2 = \id_V$. Hence $\sigma$ is a descent datum on $V$, and $V^{\sigma=\id}$ is its underlying real vector space. Hence, the functor $\omega$ is faithful and monoidal, so that it is a fiber functor. The last part of the proposition follows from the fact that a real quadratic form $q$ on a real vector space $W$ is positive definite if and only if the associated hermitian form $u,v \mapsto q(u, \overline v)$ is positive definite on $W_\bC$.
\end{proof}

The following proposition is the combination of results of \cite{mil02} and \cite{del82}.

\begin{prop} \label{prop:arch-functor-lefschetz}
  Let $k$ be a finite field. There exists a $\otimes$-functor
  \[\Rinf: \Lmot(k) \to \Isoc_{\bR},\]
  which is compatible with weights, and such that for any abelian variety $A$ with ample line bundle $\mathcal L$, bilinear form
  \[ \Rinf(\eta_{\mathcal L}): \Rinf(\frk P^n(A)) \otimes \Rinf(\frk P^n(A)) \to \bR(-n), \]
  is positive definite.
  Moreover, for any $p$-adic field $K$ with residue field $k$, and any
  \[N \in \Lmot^{\mathrm{CM}}(\mathcal O_K)\]
  with a bilinear form $\eta$ such that $R_{H,\iota}(q)$ is a polarisation of Hodge structure for some $\iota: K \to \bC$, then $\Rinf(R_{\frkp}(\eta))$ is positive definite.
\end{prop}

\begin{remark}
  In the following proof, we make use of results of \cite{mil02} in the form they were in the v2 version on arXiv, that we will cite as \cite{mil01}.
\end{remark}

\begin{proof}
  According to Milne \cite[Prop 1.5]{mil01}, there is a unique polarization $\Pi$ on $\Lmot(\Fpbar)$ which is compatible with the reduction functor
  \[R_p:\Lmot^{\mathrm{CM}}(\overline {\bQ}) \to \Lmot(\Fpbar), \]
  for a fixed $p$-adic prime of $\overline{\bQ}$.
  That is, if one has a $\CM$ Lefschetz motive $M$ over $\overline{\bQ}$, with a bilinear form $\eta$ such that $R_H(\eta)$ is a polarization of Hodge structure, then $R_p(\eta)$ is a polarization in the sense of $\Pi$. By \cite[Chap 2, Thm 5.20]{del82}, one gets a tensor functor compatible with the Tate--triple structures
  \[\Rinf: \Lmot(\Fpbar) \to \Isoc_\bR\]
  such that $\eta$ on $M \in \Lmot(\Fpbar)$ is positive in the sense of $\Pi$ if and only if $\Rinf(\eta)$ is positive definite in the sense of \Cref{def:positive-definite}. More concretely \cite[\S 5]{del82}, $\Rinf$ being compatible with the Tate structures, means it is compatible with weights and sends the Tate object to the Tate object.

  One gets a functor
  \[\Rinf: \Lmot(\bF_q) \to \Isoc_\bR \]
  by composing with the extension of scalars $\Lmot(\bF_q) \to \Lmot(\Fpbar)$.

  According to \cite[Prop 1.5]{mil01}, geometric Weil forms are positive for $\Pi$. We deduce from \cite[Remark 4.8]{mil01} that $\eta_{\mathcal L}$ is positive for $\Pi$ for any abelian variety $A$ over $\bF_q$ with ample line bundle $\mathcal L$. Hence $\Rinf(\eta_{\mathcal L})$ is positive definite in the sense of \Cref{def:positive-definite}.
  
  Moreover, consider $N \in \Lmot^\CM(\mathcal O_K)$ where $\mathcal O_K$ is the ring of integers of a $p$-adic field, endowed with a bilinear form $\eta$. Assume that $R_{H, \iota}(\eta)$ is a polarization for some $\iota: K \hookrightarrow \bC$. By definition, $N$ is a direct summand of the motive of an abelian variety $A$ of $\CM$-type. Such abelian varieties are defined over number fields, that is there exists a number field $D \subset K'$ and an abelian variety $B$ over $D$ such that
  \[A \times_K K' \cong B \times_D K',\]
  where $K'\subset \bC$ is a finite extension of $K$, which we will assume for simplicity to be $K$. We denote by $\sigma$ the inclusion $D \subset \bC$.
  
  Moreover, we can assume that the geometric endomorphisms of $B$ are defined over $D$, so that $B$ is also of $\CM$-type. The motive $N$ is a direct summand of $\h^\star(A)$, given by the image of a projector $\pi$. We get a motive $M$ over $D$, by defining it as the image of $\pi$ in $\h^\star(B)$. Hence we have
  \[N \cong M \times_D K. \]
  Moreover, the bilinear form $\eta$ on $N$ gives rise to a bilinear form $\eta'$ on $M$, and we have $R_{H, \sigma}(\eta') = R_{H, \iota}(\eta)$, which is a polarization of Hodge structure. Hence $R_p(\eta')_{\Fpbar}$ is positive for $\Pi$ as a bilinear form in $\Lmot(\Fpbar)$. But $R_p(\eta)_{\Fpbar} = R_p(\eta')_{\Fpbar}$, so that $\Rinf(\eta)$ is positive definite.
\end{proof}

\begin{cor} \label{cor:rinf-fiber-fun}  
  Let $k$ be a finite field, and $\Lmot(k)^{\text{even}}$ be the category of Lefschetz motives of even weight. The composition $\omega_\infty = \omega \circ \Rinf$, where $\Rinf$ is the functor from \Cref{prop:arch-functor-lefschetz} and $\omega$ the functor from from \Cref{prop:fiber-even-isoc}, is a fiber functor
  \[ \omega_{\infty} : \Lmot(k)^{\text{even}}  \to \Vect_{\bR}.\]
\end{cor}

\begin{proof}
  As the functor $\Rinf$ is compatible with weights, it sends motives of even weights to objects of $\Isoc_\bR$ of even weight. Hence $\omega$ precomposed with
  \[\Rinf: \Lmot(k)^{\text{even}} \to \Isoc_\bR^{\text{even}},\]
  defines a real valued fiber functor on Lefschetz motives of even weights.
\end{proof}

\begin{remark}
   In \cite[Prop 4.4.3.1]{saa72}, Saavedra proves a statement which is similar to the previous corollary for Grothendieck motives over finite fields, but conditionnaly on the standard conjectures.
\end{remark}

The following theorem is a purely tannakian statement, which will be useful to compare quadratic forms.

\begin{theorem} \label{thm:tann-main}
  Let $\mathcal T$ be a tannakian category over $\bR$, endowed with two fiber functors
  \[\omega, \omega': \mathcal T \to \Vect_\bR.\]
  Let $M$ be an object of $\mathcal T$, endowed with two quadratic forms
  \[ \eta, \eta': {\Sym}^2 M \to \bUn.\]
  Assume that $\omega(\eta)$ and $\omega(\eta')$ are positive definite. Then $\omega'(\eta)$ and $\omega'(\eta')$ have the same signature.
\end{theorem}

\begin{proof}
  As $\omega(\eta)$ and $\omega(\eta')$ are positive definite, they are nondegenerate. This means that the maps
  \[r_\eta, r_{\eta'}: M \to M^*\]
  induced by $\eta$ and $\eta'$ become isomorphisms after applying $\omega$, hence they are already isomorphisms. Hence $\eta$ and $\eta'$ are nondegenerate. Let $u$ be the unique endomorphism of $M$ such that
  \[\eta(x, uy) = \eta(ux, y) = \eta'(x,y). \]
  More precisely, this means that the following diagram is commutative
  \begin{center}
    \begin{tikzcd}
      M \otimes M \arrow[d, "\id \otimes u"] \arrow[r, "u \otimes \id"] \arrow[rd, "\eta'"] & M \otimes M \arrow[d, "\eta"] \\
      M \otimes M \arrow[r, "\eta"]& M
    \end{tikzcd}
  \end{center}
  Concretely, $u = r_{\eta}^{-1} \circ r_{\eta'}$.

  Let $\chi_u(X)$ denote the characteristic polynomial of $u$ (\Cref{def:det-charpoly}). By \Cref{prop:charpoly-fiberfunctor}, we have that
  \[\chi_{\omega(u)}(X)= \chi_u(X) = \chi_{\omega'(u)}(X). \]
  We have that $\omega(\eta)$ and $\omega(\eta')$ are positive definite, and that
  \[\omega(\eta')(x,y) = \omega(\eta)(x, \omega(u) y)\]
  for all $x,y \in \omega(M)$. We deduce that $\omega(u)$ is a positive autoadjoint endomorphism of $\omega(M)$ with respect to the scalar product $\omega(\eta)$. Hence the eigenvalues of $\omega(u)$ are positive, so that the roots of $\chi_u(X)$ are positive.

  Hence the roots of $\chi_{\omega'(u)}(X)$ are also positive, and we deduce by \Cref{lem:constant-signature} that $\omega'(\eta)$ and $\omega'(\eta')$ have the same signature. Indeed,
  \[\omega(\eta')(x,y) = \omega'(\eta)(x, \omega'(u)y) \]
  for all $x, y \in \omega'(M)$, $\omega'(u)$ is autoadjoint with respect to $\omega'(\eta)$ and has positive eigenvalues.
\end{proof}

\begin{lemma} \label{lem:constant-signature}
  Let $V$ be a $\bR$-vector space, $\eta$ be a nondegenerate quadratic form, and $u$ an endomorphism of $V$ which is autoadjoint with respect to $\eta$ and has positive eigenvalues. Then $\eta$ has the same signature as
  \[x, y \mapsto \eta(x, uy). \]
\end{lemma}
\begin{proof}
  Consider $u_\lambda = \lambda u + (1-\lambda) \id_V$ for $\lambda \in [0,1]$, and
  \[\eta_\lambda: x,y \mapsto \eta(x, u_\lambda y).\]
  $u_\lambda$ is autoadjoint with respect to $\eta$, so that $\eta_\lambda$ is a symmetric bilinear form. Because the eigenvalues of $u$ are positive, so are those of $u_\lambda$. Hence $u_\lambda$ is invertible, and $\eta_\lambda$ is  nondegenerate.

  Hence, $\lambda \mapsto \eta_\lambda$ is a continuous path between $\eta$ and $\eta_1$ in the space of nondegenerate symmetric bilinear forms. As the space of nondegenerate symmetric bilinear forms with a fixed signature $(s_+, s_-)$ is open for any $(s_+, s_-)$, it must also be closed, as its complement is a finite union of open subsets. Hence the signature of $\eta_\lambda$ is constant, and $\eta_0$ has the same signature as $\eta_1$.
\end{proof}

Here is a corollary of \Cref{thm:tann-main}, which we will be using in the proof of \Cref{thm:main}.
\begin{cor} \label{cor:two-quad-forms}
  Let $N \in \Lmot(\bF_q)$ be a Lefschetz motive of weight $2n$, with
  \[ \rk(N) = \dim_\bQ (\Rz(N(n))).\]
  Let
  \[\eta, \eta': N(n) \otimes N(n) \to \bUn \]
  be two quadratic forms, such that $\Rinf(\eta)$ and $\Rinf(\eta')$ are positive definite. Then $\Rz(\eta)$ and $\Rz(\eta')$ have the same signature.
\end{cor}

\begin{proof}
  As $N$ has weight $2n$, $N(n)$ has weight $0$. As $\rk(N) = \dim_{\bQ}(\Rz(N(n)))$, we have that $\Rz$ defines a fiber functor
  \[ \Rz: \langle N(n) \rangle^\otimes \to \Vect_{\bQ}. \]
  By \Cref{cor:rinf-fiber-fun}, we have that $\Rinf$ also induces a fiber functor $\omega_{\infty}$ on $\langle N(n) \rangle^\otimes$. We conclude by \Cref{thm:tann-main}.
\end{proof}

\begin{remark}
  The assumption $\rk(N) = \dim_{\bQ} (\Rz(N(n)))$ does not force $N$ to be $\bUn(-n)^{\oplus \rho}$. This would be the case if $N$ was an homological motive in $\Mot(\bF_q)$. Indeed, for a Lefschetz motive to be isomorphic to $\bUn(-n)^{\oplus \rho}$, it has to contain only Lefschetz classes. In other words, the assumption $\rk(N) = \dim_{\bQ} (\Rz(N(n)))$ says that $N$ contains a lot of algebraic classes, but not necessarily Lefschetz classes.
\end{remark}

\section{Decomposition of abelian motives} \label{sec:decomp-abel-motiv}
The proof of the standard conjecture of Hodge type for abelian fourfolds in \cite{anc21} involves a decomposition result for abelian motives. The same decomposition was used in \cite{clo99} to prove cases of the $\mathrm{hom = num}$ conjecture. It is quite explicit in the sense that one can write down the algebraic cycles giving rise to the direct sum. In this article, the decompositions are finer and less explicit, they are abstracted through tannakian methods.

\begin{definition} \label{def:odd-liftable}
  We say that a motive $M \in \Lmot(\bF_q)$ is odd liftable if there exist an odd integer $m$, a $p$-adic field $K$ with ring of integers $\mathcal O_K$, residue field $\bF_{q^a}$, maximal ideal $\frkp$, and a motive $N \in \Lmot^{\mathrm{CM}}(\mathcal O_K)$ whose reduction satisfies $R_\frkp(N) \cong M_{\bF_{q^a}}^{\oplus m}$.
\end{definition}

For a simple abelian variety $A/\bF_q$, we consider its multiplicity $m$, its Frobenius rank $r$ and its center $F$ as in \Cref{not:ab-var}. The notion of enriched eigenvalues is defined in \Cref{def:enriched-eigenvalue-group}.

\begin{theorem} \label{thm:exomot}
  Let $A$ be a simple abelian variety over $\bF_q$, of dimension $g$, Frobenius rank $r$ and odd multiplicity $m$. Assume that
  \begin{enumerate}
  \item the geometric endomorphisms of $A$ are defined over $\bF_q$,
  \item $r \geqslant \frac{g}{m}-1$,
  \item there exists a totally real field $D$ disjoint from $F$ such that $A$ has $\CM$ by $L = D \cdot F$, i.e. $\deg D = m$ and $D \cdot F \hookrightarrow \End(A)$.
  \end{enumerate}
  
  Let $M$ be a simple Lefschetz motive $M$ of weight $2n$ with enriched eigenvalue $\lambda \not = [q]^n$ and $\rho(\lambda) = q^n$.
  Then $M$ has rank $2$ and is odd liftable.
\end{theorem}
\begin{proof}
  Let $A$ be such an abelian variety, $F$ be the center of $\End^0(A)$, $\tilde F$ be a Galois closure and
  \[\pi_1, \dots, \pi_{g/m}, \pi_{g/m+1} \dots, \pi_{2g/m}\]
  be the eigenvalues of Frobenius in $\tilde L$, such that $\overline \pi_i = \pi_{g/m+i}$. Let us first describe the group $H$ of $\lambda \in \Eig(A)$ such that $\rho(\lambda) = q^n$ for some $n$. Consider the $\Gal(\tilde F/\bQ)$-equivariant exact sequence
  \[ 0 \rightarrow H \rightarrow \Eig(A) \overset{\overline \rho}{\rightarrow} \tilde F^\times/q.\]
  By definition, $r$ is the rank of the image of $\overline \rho$. Moreover, $\Eig(A)$ has rank $g/m + 1$.

  Hence $H$ has rank $\frac{g}{m}+1 - r$, which is at most $2$ by the numerical assumption. We always have $[q] \in H$, hence $H$ either has rank $1$ or $2$. Moreover, $\Eig(A)$ being torsion free, so is $H$. Assume that $H$ has rank $1$, then any $\lambda$ with $\rho(\lambda) = q^n$ for some $n$ is itself $[q]^n$. In this case, there is no motive $M$ satisfying the assumptions.

  Otherwise, there is $\mu = [q]^{\alpha_0} [\pi_1]^{\alpha_1} \dots [\pi_{g/m}]^{\alpha_{g/m}}$ such that $H = [q]^{\bZ} \cdot \mu^{\bZ}$. Up to changing $\mu$, we will assume $\alpha_0=0$. Consider $\overline H = H/[q]$, which is generated by $\overline \mu$, and is isomorphic to $\bZ$. The group $\Gal(\tilde F/\bQ)$ acts on $\overline H$, by a sign as $\mathrm{GL}_1(\bZ) = \{\pm 1\}$. Moreover, $\Gal(\tilde F/ \bQ)$ permutes transitively $\{\pi_1, \dots, \pi_{2g/m}\}$. Hence for any $i$, $\alpha_i = \pm \alpha_1$ and up to changing $\pi_i$ with $\overline \pi_i$, we can assume $\alpha_i = \alpha_1=\alpha$. Hence $\mu = ([\pi_1] \dots [\pi_{g/m}])^\alpha$ and
  \[H = [q]^\bZ \cdot ([\pi_1] \dots [\pi_{g/m}])^{\alpha\bZ}.\]
  As $\Gal(\tilde F/\bQ)$ acts on $\overline H$ by a sign, we get that for each $\sigma \in \Gal(\tilde F/\bQ)$,
  \[ \sigma \mu =
    \begin{cases}
      \mu \; \; \; \text{ or,} \\
      [q]^{\alpha g/m} \mu^{-1}.
    \end{cases}
  \]
  Now, let $M$ be a motive as in the statement, with enriched eigenvalue $\lambda$. As $\rho(\lambda) = q^n$ for some $n$, we have $\lambda \in H$, so that
  \[\lambda = [q]^i ([\pi_1] \dots [\pi_{g/m}])^{\alpha j}.\]
  As $\rho(\lambda) = q^n$, we have by \Cref{prop:motive-rank} that $\dim( M)$ is the cardinal of the orbit of $\lambda$ under $\Gal(\tilde F/\bQ)$. As $\lambda \not = [q]^n$, we see that $j \not = 0$. Hence the orbit of $\lambda$ is
  \[\{ [q]^i \mu^j, [q]^{i+\frac{jg}{m}} \mu^{-j}\}, \]
  and $M$ has rank $2$.

  Let $D$ be a totally real field of degree $m$ disjoint from $F$, with a map $D \cdot F=L \to \End^0(A)$. Let $\tilde L$ be a Galois closure of $L$ and $G = \Gal(\tilde L/\bQ)$. Now, by \cite[Théorème 2]{tat71} there exists a $p$-adic field $K$ with rings of integers $\mathcal O_K$, residue field $\bF_{q^a}$, and an abelian scheme $\mathcal A$ over $\mathcal O_K$ with complex multiplication by $L$ such that $\mathcal A_{\frkp}$ is isogenous to $A_{\bF_{q^a}}$, moreover we can assume that the geometric endomorphisms of $\mathcal A$ are defined over $K$. According to \cite[Prop 2.5]{mil99} and \Cref{prop:simple-object-tannaka-cat}, the simple objects of $\langle \mathcal A \rangle$ are parametrized by $G$-orbits in
  \[ X^\star(L^\Psi) = \frac{\{f: \Psi \to \bZ\}}{ \{f \; | \; f = \iota f \text{ and } \sum_{\psi \in \Psi} f(\psi) = 0 \}},\]
  where $\Psi$ is the Galois orbit of $\CM$-types associated to $\mathcal A_K$. Let us fix an embedding $\tilde L \to \overline K$, then by \cite[Proposition 2.2]{mil99}, we have that $\Psi \cong \Hom(L, \tilde L)$. To a character $\chi: L \to \tilde L$ of $L$, one can associate its characteristic function
  \begin{align*}
    f_\chi: \varphi \mapsto
    \begin{cases}
      1 \text{ if } \chi = \varphi \\
      0 \text{ otherwise.}
    \end{cases}
  \end{align*}
  As $D$ and $F$ are disjoint, $D\cdot F = D \otimes F$, hence there is a $\Gal(\tilde L/\bQ)$-equivariant identification
  \[ \Hom(L, \tilde L) \cong \Hom(D, \tilde L) \times \Hom(F, \tilde L). \]
  Moreover, one has $\Hom(F, \tilde L) \cong \{\pi_1, \dots, \pi_{g/m}, \overline{\pi_1}, \dots, \overline{\pi_{g/m}}\}$. Fix a embedding $\lambda \in \Hom(D, \tilde L)$, and see $D$ as a subfield of $\tilde L$. Consider the characters $\chi_1, \dots, \chi_{g/m}$ of $L$ which correspond to $(\lambda, \pi_i)$ for $1 \leqslant i \leqslant g/m$. One has $\chi_i(\pi)=\pi_i$. By \cite[Theorem 5.4]{mil99}, reduction modulo $\frkp$ of the motive associated to
  \[ f = \alpha j f_{\chi_1} + \dots + \alpha j f_{\chi_{g/m}} \]
  has $\mu^j = [\chi_1(\pi)]^{\alpha j} \dots [\chi_{g/m}(\pi)]^{\alpha j}$ as enriched eigenvalue. Moreover,
  \[\Fix_{\Gal(\tilde L/\bQ)}([f]) = \Stab_{\Gal(\tilde L/\bQ)}(\{\chi_1, \dots, \chi_{g/m}\}),\]
  because $f = \sigma \cdot f + g$ with $\iota g = g$ if and only if
  \[\{\chi_1, \dots, \chi_{g/m}\} = \{\sigma \cdot \chi_1, \dots, \sigma \cdot \chi_{g/m} \}.\]
  Moreover,
  \[ \Stab_G(\{\chi_1, \dots, \chi_{g/m}\})\]
  has index $2$ in $\Gal(\tilde L/D)$, which has index $m$ in $G=\Gal(\tilde L/\bQ)$. Denote by $S$ the $\Gal(\tilde L/\bQ)$-orbit of $[f]$ and $N$ the simple object associated, we conclude that $\dim(N) = \#S = 2m$. Moreover, $R_{\frkp}N(-i)$ has $[q]^i \mu^j$ as enriched eigenvalue, and for dimension reasons we have
  \[R_\frkp N(-i) \cong M_{\bF_{q^a}}^m.\]
  Hence $M$ is odd liftable.
\end{proof}

\begin{prop} \label{prop:simple-lef-mot}
  Let $A$ be a simple abelian variety over $\bF_q$, of dimension $g$, Frobenius rank $r$ and odd multiplicity $m$. Assume that $r \geqslant \frac{g}{m}-1$ and that there exists a totally real field $D$ disjoint from $F$ such that $A$ has $\CM$ by $D \cdot F$. Then for any simple Lefschetz motive $M\in \langle A \rangle^\otimes$ pure of weight $2n$ with enriched eigenvalue $\lambda$, we have $3$ possibilities
  \begin{enumerate}
  \item either $\lambda = [q]^n$, in which case $M=\bUn(-n)$
  \item or $\lambda \not= [q]^n$ and $\rho(\lambda) = q^n$, in which case $M$ has rank $2$ and is odd liftable.
  \item or $\rho(\lambda) \not = q^n$, in which case $\Rz(M(n)) = 0$.
  \end{enumerate}
\end{prop}
\begin{proof}
  If $\lambda = [q]^n$, then $M = \bUn(-n)$ by \Cref{prop:lefmot-classification}, because $\bUn(-n)$ also has $[q]^n$ as enriched eigenvalue. If $\lambda \not = [q]^n$, and $\rho(\lambda)=q^n$ then $M$ has rank $2$ and is odd liftable by \Cref{thm:exomot}.

  It remain to treat the case where $\rho(\lambda) \not = q^n$. According to \Cref{rem:cycle-eigen}, algebraic classes in $M(n)$ are homologically trivial in this case, hence numerically trivial. We conclude that $\Rz(M(n))=0$.
\end{proof}

\begin{prop} \label{prop:decomp-main}
  Let $A$ be a simple abelian variety over $\bF_q$, of dimension $g$, Frobenius rank $r$ and odd multiplicity $m$. Let $P$ be a Lefschetz motive in $\langle A \rangle ^\otimes$, pure of weight $2n$. Assume that $r \geqslant \frac{g}{m}-1$ and that there exists a totally real field $D$ disjoint from $F$ such that $A$ has $\CM$ by $D \cdot F$. Then there are submotives $L, (E_i)_{1 \leqslant i \leqslant d}$ and $T$ of $P$ such that
  \begin{enumerate}
  \item $P = L \oplus E_1 \oplus \dots \oplus E_d \oplus T$
  \item $L$ is isotypical of type $\bUn(-n)$
  \item $E_i$ is isotypical of type $M_i$, where $M_i$ has rank $2$ and is odd liftable,
  \item $R_Z(T(n)) = 0$.
  \end{enumerate}
  Moreover, the motives $M_i$ are nonisomorphic and $M_i(n)^\star \cong M_i(n)$.
\end{prop}

\begin{proof}
  Let $P$ a motive in $\langle A \rangle^\otimes$, which is pure of weight $2n$. As $\langle A \rangle^\otimes$ is a semisimple tannakian category, we have the following isotypical descomposition
  \[ \bigoplus_{i = 0}^s M_i \otimes \Hom(M_i, P) \xrightarrow{\sim} P,\]
  where $M_0, \dots, M_s$ are nonisomorphic simple objects. As $P$ is pure of weight $2n$, $M_i$ is also pure of weight $2n$. Let $\lambda_i$ be the enriched eigenvalue associated to $M_i$ via \Cref{prop:lefmot-classification}, it is also of weight $2n$.
  
  According to \Cref{prop:simple-lef-mot}, there are three possibilities for each $M_i$. Without loss of generality, we can assume that
  \begin{align*}
    \lambda_0 &= [q]^n \\ \rho(\lambda_1) &= \dots = \rho(\lambda_d)= q^n \\ \rho(\lambda_i) &\not = q^n \text{ for } i > d.
  \end{align*}
  Hence the following subobjects of $P$ will satisfy the desired properties
  \begin{align*}
    L &= M_0 \otimes \Hom(M_1, P) \\
    E_i &= M_i \otimes \Hom(M_i, P) \text{ for } 1 \leqslant i \leqslant d \\
    T &= \bigoplus_{i > d} M_i \otimes \Hom(M_i, P).
  \end{align*}
  By definition, the $M_i$ are not isomorphic. Moreover, for all $i$ we have that $M_i(n)^* \cong M_i(n)$.
\end{proof}

\section{Application to positivity of intersections}\label{sec:appl-posit-inters}

This section is dedicated to the proof of \Cref{thm:main}. First, \Cref{lem:scht-lef} abstracts from the situation of an intersection pairing on primitive parts to a general pairing satisfying some conditions with respect to the functor $\Rinf$ of \Cref{prop:arch-functor-lefschetz}. Then, using \Cref{prop:decomp-main}, we have a case disjonction, and the two main cases are delt with in \Cref{lem:positivity-lef} and \Cref{lem:positivity-exo}. In the proof of \Cref{lem:positivity-exo} we reduce the case of an isotypical motive $E$ to the case of its simple constituent. The case of simple exotic motives of rank $2$ is treated in \Cref{lem:quad-existence-exo}, it is the main case of the proof, in which \Cref{cor:two-quad-forms} is used crucially.

For a simple abelian variety $A/\bF_q$, we consider its multiplicity $m$ and its rank of Frobenius as in \Cref{not:ab-var}.

We make use of the functors $\Rinf$ from \Cref{prop:arch-functor-lefschetz} and $\Rz$ from \Cref{conv:lef-mot}.

\begin{proof}[Proof of \Cref{thm:main}]
  Let $A$ be a simple abelian variety of dimension $g$ over $\Fpbar$ as in the theorem.

  Fix $d \in \mathbb N$, and let us prove the standard conjecture of Hodge type for $A^d$. According to \Cref{prop:cyc-mot-comp}, it suffices to prove for each $n\in \bN$ and each ample divisor $\mathcal L$ on $A^d$ that the quadratic form
  \[ \eta_{\mathcal L}: \frk P^{2n}_{\mathcal L}(A^d) \otimes \frk P^{2n}_{\mathcal L}(A^d) \to \bUn(-2n)\]
  is such that $\Rz(\eta_{\mathcal L}(2n))$ is positive definite on $\Rz(\frk P^{2n}_{\mathcal L}(A^d)(n))$. Let us fix $n \in \bN$, and $\mathcal L$ an ample divisor on $A^d$. We will be working with Lefschetz motives, and we write
  \[P = \frk P^{2n}_{\mathcal L}(A^d)_\lef \text{ and } \eta = \eta_{\mathcal L}(2n). \]

  We have $\Rinf(\eta)>0$ by \Cref{prop:arch-functor-lefschetz}. Hence according to \Cref{lem:scht-lef}, $\Rz(\eta)$ is positive definite, which concludes the proof.
\end{proof}

\begin{lemma} \label{lem:scht-lef}
  Let $A$ be an abelian variety over $\Fpbar$ which is simple of dimension $g$, with odd multiplicity $m$ and Frobenius rank $r$ (c.f. \Cref{rem:mult-rank-fpbar}). Assume that
  \[r \geqslant \frac{g}{m} - 1,\]
  and that there exists a totally real field $D$ disjoint from $F$ such that $A$ has $\CM$ by $D \cdot F$.
  Then for any Lefschetz motive $P$ in $\langle A \rangle^\otimes$, which is pure of weight $2n$, and any quadratic form $\eta$ on $P(n)$ such that $\Rinf(\eta)$ is positive definite, then $\Rz(\eta)$ is positive definite.
\end{lemma}
\begin{proof}
  Let $A, P$ and $\eta$ be as in the statement. \Cref{prop:decomp-main} is about Lefschetz motive over $\bF_q$, hence we first need to find models for $A, P$ and $\eta$ over a finite field.

  As $A$ is an abelian variety, there exists a finite field $k_0$ and $A_0$ over $k_0$ an abelian variety such that $A= A_0 \times_{k_0} \Fpbar$, and such that $A_0$ is geometrically simple with multiplicity $m$ and Frobenius rank $r$, and whose geometrical endomorphisms are defined over $k_0$. Then, $P$ is a direct summand of some power of $\frk h^1(A)^{\otimes 2n}$. The projector defining $P$ is a Lefschetz cycle on a power of $A$, and it is defined over $k_0$ because the endomorphisms of $A$ are defined over $k_0$. This projector over $k_0$ defines a direct summand $P_0$ of some power of $\frk h^1(A_0)^{\otimes 2n}$, such that
  \[P_0\times_{k_0} \Fpbar = P.\]
  Moreover, $\Rz(P_0(n)) \subset \Rz(P(n))$ is the subset of numerical cycles in $P$ defined over $k_0$. As $\Rz(P(n))$ is a finite dimensional $\bQ$-vector space, we can assume that a finite set of numerical cycles which generates $\Rz(P(n))$ is defined over $k_0$, so that $\Rz(P_0(n)) = \Rz(P(n))$. We can also assume that the Lefschetz cycles defining $\eta$ are defined over $k_0$, so that we have a a quadratic form $\eta_0$ on $P_0(n) \in \Lmot(k_0)$.

  Now, by \Cref{prop:decomp-main} we can write $P_0$ as
  \[P_0 = L \oplus E_1 \oplus \dots \oplus E_d \oplus T, \]
  where $L$ is isotypical of type $\bUn(-n)$, $E_i$ is isotypical of type $M_i$, with $M_i$ odd liftable (in the sense of \Cref{def:odd-liftable}) of rank $2$, and $T$ is such that $\Rz(T) = 0$. As $M_i(n)^\star \cong M_i(n)$, and the motives $M_i$ are nonisomorphic, we have that $L$ and the $E_i$'s are orthogonal with respect to $\eta_0$. This shows that the decomposition
  \[ \Rz(P_0(n)) = \Rz(L(n)) \oplus \Rz(E_1(n)) \oplus \dots \oplus \Rz(E_d(n)) \oplus 0\]
  is orthogonal with respect to $\Rz(\eta_0)$. Hence, to prove that $\Rz(\eta_0)$ is positive definite, it suffices to prove that $\Rz({\eta_0}_{| L(n)})$ and $\Rz({\eta_0}_{|E_i(n)})$ are positive definite.

  On the other hand, $\Rinf(\eta_0) = \Rinf(\eta)$ is positive definite. Consequently, $\Rinf({\eta_0}_{|L})$ and $\Rinf({\eta_0}_{|E_i})$ are positive definite. By \Cref{lem:positivity-lef} and \Cref{lem:positivity-exo} we conclude that $\Rz({\eta_0}_{|L})$ and $\Rz({\eta_0}_{|E_i})$ are positive definite as well.
\end{proof}
\begin{lemma} \label{lem:positivity-lef}
  Let $\Rinf$ be the functor from \Cref{prop:arch-functor-lefschetz} and $L$ be a Lefschetz motive which is isotypical of type $\bUn(-n)$ with $n\in \bN$. For any quadratic form $\eta$ on $L(n)$, if $\Rinf(\eta)$ is positive definite, then $\Rz(\eta)$ is positive definite.
\end{lemma}
\begin{proof}
  As $L$ is isotypical of type $\bUn(-n)$, $L(n)$ is isotypical of type $\bUn$. That is, $L(n)$ can be seen as a $\bQ$-vector space. Hence $\Rinf(L(n)) = \Rz(L(n))$, and $\Rz(\eta) = \Rinf(\eta)$ is positive definite.
\end{proof}

\begin{lemma} \label{lem:positivity-exo}
  Let $\Rinf$ be the functor from \Cref{prop:arch-functor-lefschetz}, and $E$ be a Lefschetz motive of weight $2n$ over $\bF_q$, which is isotypical of type $M$, where $M$ is odd liftable of rank $2$. For any quadratic form $\eta$ on $E(n)$, if $\Rinf(\eta)$ is positive definite, then $\Rz(\eta)$ is positive definite.
\end{lemma}

\begin{proof}
  Let $E$ and $M$ and $\eta$ as in the statement. If $\Rz(M(n)) = 0$, then $\Rz(E(n)) = 0$ and there is nothing to prove. As $M$ is simple, $\End(M)$ is a field. Moreover,
  \[\End(M) \otimes_{\bQ} \overline{\bQ} = \End(M_{\overline{\bQ}}),\]
  and $M_{\overline{\bQ}}$ is not simple (as the simple objects are characters) so that $\dim(\End(M)) > 1$. This shows that $\Rz(M(n))$ cannot have dimension $1$, because the field $\End(M)$ is acting on it. Hence, $\dim_\bQ \Rz(M(n)) = 2$ and $M$ is odd liftable. Hence by \Cref{lem:quad-existence-exo} there exists a quadratic form $\eta'$ on $M_{\bF_{q^a}}(n)$ for some $a$ such that $\Rz(\eta')$ and $\Rinf(\eta')$ are positive definite. Let us fix an isomorphism
  \[E \cong M^{\oplus d}.\]
  We have two quadratic forms $\eta$ and $(\eta')^{\oplus d}$ on $E_{\bF_{q^a}}(n)$, which are sent to positive definite forms by $\Rinf$. Then by \Cref{cor:two-quad-forms}, we get that $\Rz(\eta)$ and $\Rz(\eta')^{\oplus d}$ have the same signature. As $\Rz(\eta')$ is positive definite, we deduce that $\Rz(\eta)$ is also positive definite, which concludes the proof.
\end{proof}

\begin{prop} \label{lem:quad-existence-exo}
  Let $\Rinf$ be the functor from \Cref{prop:arch-functor-lefschetz} and $M$ be an irreducible motive of rank $2$, of even weights in $\Lmot(\bF_q)$ such that
  \begin{enumerate}[label=(\roman*)]
  \item $\dim \Rz(M) = 2$
  \item $M$ is odd liftable.
  \end{enumerate}
  Then there exists a quadratic form $\eta$ on $M_{\bF_{q^a}}$, for some $a$, such that $\Rz(\eta)$ and $\Rinf(\eta)$ are positive definite.
\end{prop}
\begin{proof}
  Let $M$ as in the statement, and because $M$ is odd liftable (\Cref{def:odd-liftable}) we have a $p$-adic field $K$ with ring of integers $\mathcal O_K$, residue field $\bF_{q^a}$ and $N \in \Lmot^{\mathrm{CM}}(\mathcal O_K)$ with $\CM$ by $L$ such that $R_\frkp(N) \cong M_{\bF_{q^a}}^{\oplus m}$ with $m$ odd. Up to taking a finite extension, we can assume that $N$ is geometrically simple. By \Cref{lem:motive-polarization}, there exists a quadratic form $\eta_0$ on $N$ such that $R_H(\eta_0)$ is a polarization of Hodge structure and for every $z\in L$, the adjoint of $z$ with respect to $\eta_0$ is $\overline z$. Hence by \Cref{prop:arch-functor-lefschetz}, $\Rinf(\eta_{0,\frkp})$ is positive definite.

  Let $f: M \to N_\frkp$ be any nonzero map, and consider $\eta$ the quadratic form on $M$ obtained by restraining $\eta_{0,\frkp}$ to $M$ via $f$. Then $\Rinf(\eta)$ is the restriction of $\Rinf(\eta_{0,\frkp})$, hence it is positive definite.

  The quadratic forms $\eta^{\oplus m}$ and $\eta_{0,\frkp}$ on $M^{\oplus m} \cong N_\frkp$, are sent to positive definite forms by $\Rinf$. Hence $\Rz(\eta)^{\oplus m}$ and $\Rz(\eta_{0,\frkp})$ have the same signature by \Cref{cor:two-quad-forms}. We have $3$ possibilities for $\Rz(\eta)$:
  \begin{enumerate}[label=\arabic*)]
  \item If it has signature $(2,0)$ then $\Rz(\eta_{0,\frkp})$ has signature $(2m, 0)$,
  \item if it has signature $(1,1)$ then $\Rz(\eta_{0,\frkp})$ has signature $(m,m)$,
  \item if it has signature $(0,2)$ then $\Rz(\eta_{0,\frkp})$ has signature $(0,2m)$.
  \end{enumerate}
  Let $(s_+, s_-)$ be the signature of $\Rz(\eta_{0,\frkp})$, according to \Cref{thm:ancona-marmora} we have $4| s_-$. As $m$ is assumed to be odd, the only possibility is that $\Rz(\eta)$ is positive definite, which concludes the proof. \end{proof}

\printbibliography

\end{document}